\documentclass[11pt]{amsart}
\usepackage[T1]{fontenc}
\usepackage[utf8]{inputenc}
\usepackage{lmodern}
\usepackage{amsmath,amssymb,amsthm,mathtools,mathrsfs}
\usepackage[margin=1in]{geometry}
\usepackage{silence}
\usepackage{microtype}
\usepackage{enumitem}
\usepackage{needspace}
\usepackage[hidelinks]{hyperref}
\usepackage{xcolor}

\hypersetup{
 pdftitle={Beyond Mock Modularity: Elliptic Corrections for Higher Dyson Ranks},
 pdfauthor={Claudia Alfes, Ken Ono, and Ashvin Swaminathan},
 pdfsubject={Higher Dyson ranks, elliptic periodization, and theta decompositions}
}
\numberwithin{equation}{section}
\allowdisplaybreaks[1]
\newtheorem{theorem}{Theorem}[section]
\newtheorem{proposition}[theorem]{Proposition}
\newtheorem{lemma}[theorem]{Lemma}
\newtheorem{corollary}[theorem]{Corollary}
\theoremstyle{definition}
\newtheorem{definition}[theorem]{Definition}
\newtheorem{example}[theorem]{Example}
\theoremstyle{remark}
\newtheorem{remark}[theorem]{Remark}

\newcommand{\CC}{\mathbb C}
\newcommand{\HH}{\mathbb H}
\newcommand{\QQ}{\mathbb Q}
\newcommand{\ZZ}{\mathbb Z}

\newcommand{\RRm}{\mathcal R}
\newcommand{\FF}{\mathscr F}
\newcommand{\Res}{\operatorname*{Res}}
\newcommand{\qbinom}[2]{\genfrac{[}{]}{0pt}{}{#1}{#2}_{q}}

\title{Beyond Mock Modularity: Elliptic Corrections for Higher Dyson Ranks}
\author{Claudia Alfes, Ken Ono, and Ashvin Swaminathan}
\thanks{The research of the first author is funded by the Deutsche
Forschungsgemeinschaft (DFG, German Research Foundation) -- SFB-TRR 358/1
2023 -- 491392403.}
\address{Faculty of Mathematics, Bielefeld University,
Postfach 100131, 33501 Bielefeld, Germany}
\email{alfes@math.uni-bielefeld.de}
\address{Axiom Math, 124 University Avenue, Palo Alto, CA 94301, USA}
\email{ken@axiommath.ai}
\email{ashvin@axiommath.ai}
\date{}

\begin{document}
\begin{abstract}

When $m = 1$, the Dyson rank generating function is a classical bridge between
partition theory, Ramanujan's mock theta functions, and the theory of harmonic
Maass forms and nonholomorphic Jacobi forms. The rank is a statistic on
partitions, and the higher Dyson systems, for $m \geq 2$, are a natural multivariable refinement of it, combining $m$ graded
rank contributions. Unlike the classical case, these higher systems are not
expected to fit the mock-modular framework, which raises the question of what
analytic structure governs them. We show that their root-of-unity
specializations carry a hidden elliptic structure. 
For $m\ge2$, the primitive $(2m+1)$st-root
specialization of their generating function leads to a product
deformation whose associated seed satisfies a finite $q$-difference
recurrence with an explicit polynomial defect.
Using bilateral periodization we construct a meromorphic lift satisfying a twisted index-$m$ elliptic transformation law,
with at most simple poles on a prescribed lattice orbit. The poles can be removed by subtracting an explicit multiple of an Appell--Lerch kernel and a specified translation removes the twist. The resulting holomorphic finite parts admit a finite theta decomposition.  
These
results grew out of a human--AI collaboration, and the key new formulas were
formalized and machine-verified in Lean/Mathlib by AxiomProver.

\end{abstract}
\maketitle
\section{Introduction}
One of Ramanujan's most influential discoveries is his list of mock theta functions, described in his final letter to Hardy. Among them, the third order mock theta function
\[
f(q):=\sum_{n\geq 0}\frac{q^{n^2}}{(1+q)^2(1+q^2)^2\cdots (1+q^n)^2}
\]
is a central example. Ramanujan observed that such functions exhibit asymptotic behavior near roots of unity reminiscent of modular forms, even though they do not themselves satisfy the classical modular transformation laws. Zwegers' work \cite{ZwegersThesis} later explained this phenomenon by showing that mock theta functions admit nonholomorphic completions which transform as harmonic Maass forms. This discovery placed Ramanujan's examples in a broad analytic framework connecting $q$-series, modular forms, harmonic Maass forms, and Jacobi-type objects (see \cite{BFOR,BruinierFunke,OnoVisions,ZagierBourbaki} for detailed accounts).

The same function $f(q)$ also appears in a second story of Ramanujan, namely the theory of integer partitions. Let $p(n)$ denote the number of partitions of $n$. Ramanujan discovered the congruences
\[
p(5n+4)\equiv 0\pmod 5,\qquad
p(7n+5)\equiv 0\pmod 7,\qquad
p(11n+6)\equiv 0\pmod {11}.
\]
Dyson introduced the rank of a partition \cite{Dyson}, defined as the largest part minus the number of parts, as a statistic intended to explain the congruences modulo $5$ and $7$; his conjectures were subsequently established by Atkin and Swinnerton-Dyer \cite{AtkinSwinnertonDyer}. If $N(r,n)$ denotes the number of partitions of $n$ with rank $r$, then Dyson's rank generating function is
\[
\mathcal{R}_1(z;q)=\sum_{\lambda}z^{\operatorname{rank}(\lambda)}q^{|\lambda|}
       =\sum_{n\geq 0}\frac{q^{n^2}}{(zq,z^{-1}q;q)_n},
\]
where
\begin{equation*}
(a;q)_n:=(1-a)(1-aq)\cdots (1-aq^{n-1}) \qquad {\text {\rm and}} \qquad
(a,b;q)_n=(a;q)_n (b;q)_n.
\end{equation*}
We also use the standard infinite product notation
\[
(a;q)_\infty := \prod_{n\geq 0}(1-aq^n),
\qquad
(a_1,\ldots,a_r;q)_\infty := \prod_{j=1}^r (a_j;q)_\infty,
\]
whenever the products converge.

The specialization $\mathcal{R}_1(-1;q)$ is Ramanujan's function $f(q)$. More generally, the nontrivial root-of-unity specializations of $\mathcal{R}_1(z;q)$ are mock modular forms after completion \cite{BringmannOnoInvent,BringmannOno}, and their behavior at roots of unity is further governed by quantum modular forms \cite{FolsomOnoRhoades}. Thus Dyson's rank is a basic bridge between the combinatorics of partitions and the analytic theory of mock modular and harmonic Maass forms.

This paper studies the corresponding analytic problem for the higher Dyson series
\begin{equation}\label{eq:intro-Rm}
\mathcal{R}_m(z;q):=\sum_{n\geq 0}
\frac{q^{n^2}}{\prod_{j=1}^m(z^jq,z^{-j}q;q)_n}.
\end{equation}
The case $m=1$ is Dyson's rank generating function. For $m\geq 2$, the
functions $\mathcal{R}_m(z;q)$ arise from a natural combinatorial refinement
of Dyson's rank. Just as the $m=1$ series records a single rank statistic on
partitions, $\mathcal{R}_m(z;q)$ is the generating function for
\emph{uncoupled $m$-Dyson symbols} $(n;\alpha_1,\beta_1,\dots,\alpha_m,\beta_m)$,
where $n$ is a nonnegative integer and $\alpha_1,\beta_1,\dots,\alpha_m,\beta_m$
are partitions, each of largest part at most $n$. Such a symbol is weighted by
its size and by a \emph{full rank}
\[
\sum_{j=1}^m j\bigl(\ell(\alpha_j)-\ell(\beta_j)\bigr)
\]
that combines the $m$ individual ranks with weights $1,2,\dots,m$. This is the
statistic obtained when the single pair of partitions underlying the symbol
form of Dyson's rank is replaced by $m$ independent pairs graded by $j$, so
that at $m=1$ one recovers Dyson's rank generating function and the
$z\mapsto z^{-1}$ symmetry of $\mathcal{R}_m$ reflects the involution
exchanging the two partitions in each pair. In this sense the higher Dyson
series are the most direct multivariable generalization of the rank. We
give their precise definition and the resulting collapse at roots of unity in
Section~\ref{sec:dyson}.

For $m\geq 2$, we show that the higher Dyson series $\mathcal{R}_m(z;q)$ gives rise to an explicit elliptic structure.\footnote{This elliptic structure is different from the one considered in the first version of this paper.}
Our construction has two distinct ingredients. 
First, a root-of-unity specialization leads to a product deformation
whose generating function satisfies a finite recurrence. 
Second, we
apply an elliptic averaging operation to that generating function.
We compute the poles of this average and the holomorphic
part that remains after polar subtraction. The latter has $2m$ theta
components, and we give a convergent formula and a nonzero leading
term for each coefficient. We note that our averaging operation is a twisted version of the operator in
\cite[\S8.2, equation~(8.9)]{DabholkarMurthyZagier}. 

To make this precise, we begin by fixing important notation.
In analytic statements, we write
\[
\mathbb H := \{\tau\in\mathbb C : \operatorname{Im}(\tau)>0\},
\qquad q:=e^{2\pi i\tau}.
\]
Throughout the introduction, we assume that $m\geq 2$ and we set
\[
d:=2m+1.
\]
If \(\zeta_d\) is a primitive \(d\)-th root of unity, then the product in
\eqref{eq:intro-Rm} collapses to
\begin{equation*}%\label{eq:root-collapse}
    \RRm_m(\zeta_d;q)=
\sum_{n\geq0}q^{n^2}\frac{(q;q)_n}{(q^d;q^d)_n}.
\end{equation*}
The normalized specialization
\begin{equation}\label{eq:intro-normalized-root}
\frac{(q^d;q^d)_\infty}{(q;q)_\infty}\RRm_m(\zeta_d;q)
\end{equation}
is the \(x=1\), \(j=0\) member of the one-parameter deformation
\begin{equation*}%\label{eq:intro-Xj}
X_j^{(m)}(x;q):=
\sum_{n\geq0}q^{n^2+jn}
\frac{(x^dq^{d(n+1)};q^d)_\infty}{(xq^{n+1};q)_\infty}
\qquad (j\in\ZZ).
\end{equation*}
Equivalently, if we define
\begin{equation}\label{eq:G-product}
G_d(y;q):=\frac{(y^dq^d;q^d)_\infty}{(yq;q)_\infty},
\end{equation}
then we have
 \begin{equation*}%\label{eq:intro-X}
     X_j^{(m)}(x;q)=\sum_{n\geq0}q^{n^2+jn}G_d(xq^n;q).
 \end{equation*} 
 In particular, at $x=1$ and $j=0$ we recover
 \begin{equation}\label{eq:seed-specialization}
 X_0^{(m)}(1;q)=
 \frac{(q^d;q^d)_\infty}{(q;q)_\infty}\RRm_m(\zeta_d;q).
\end{equation}
We assemble these functions together with the
generating functions
\begin{equation}\label{eq:seed-def}
\mathcal F_{d,x}(t;q):=\sum_{j\geq0}X_j^{(m)}(x;q)t^j, \qquad
\mathcal G_{d,x}(u;q):=x^{-2}\mathcal F_{d,x}(xu;q).
\end{equation}
The first series converges for $|t|<1$, and its meromorphic continuation yields
\begin{equation}\label{eq:seed-fractions}
 \mathcal G_{d,x}(u;q)=
 x^{-2}\sum_{k\ge0}\frac{q^{k^2}g_{k,x}(q)}{1-xuq^k}, \qquad g_{k,x}(q):=G_d(xq^k;q).
\end{equation}

We call $ \mathcal G_{d,x}(u;q)$ the \emph{seed} of our construction; it is the function that we periodize below. Our first theorem gives a finite recurrence among the $X_j^{(m)}$, and the functional equation that it implies for the seed.

\begin{theorem}
\label{thm:finite}
Let $m\geq 2$ and $x\in\CC^\times$ with $x\notin q^{-\mathbb{Z}_{\geq 1}}$ and $C_x:=\sum_{r=0}^{d-1}x^r$. For every
$j\in\ZZ$, we have
\begin{equation}\label{eq:recurrence}
 \sum_{r=0}^{d-1}x^rX_{j+r-2}^{(m)}(x;q)
 -q^{j-1}X_j^{(m)}(x;q)=C_xG_d(x;q).
\end{equation}
There is a unique polynomial $R_{d,x}(u;q)$ of degree at most $d-2$
such that
\begin{equation}\label{eq:seed-difference}
 \begin{split}
 \mathcal G_{d,x}(qu;q)
 ={}&x^2q\frac{u^d-1}{u^{d-3}(u-1)}\mathcal G_{d,x}(u;q)
       +qu^{3-d}R_{d,x}(u;q)
    -C_xq\frac{u^2}{1-xu}G_d(x;q).
 \end{split}
\end{equation}
Writing $X_j=X_j^{(m)}(x;q)$, this polynomial is
\begin{equation}\label{eq:R-explicit}
 R_{d,x}(u;q)=q^{-1}x^{-2}X_0u^{d-3}+x^{-1}X_1u^{d-2}
 -\sum_{k=0}^{d-2}x^kX_k\sum_{a=k}^{d-2}u^a.
\end{equation}
Its constant term is $-X_0$, and its coefficients have at most a simple
pole at $q=0$. If $x\ne1$ and $x^d=1$, then $C_x=0$.
\end{theorem}

We next average the seed over its multiplicative $q$-translates to
obtain an elliptic transformation law.
To this end, for the remaining results, we let $x=e^{2\pi i\alpha}$ be a root of
unity, with $\alpha\in\QQ\cap[0,1)$ and we write $u:=e^{2\pi i w}$ for $w\in\mathbb{C}$. We define the bilateral series
\begin{equation}\label{eq:E-def}
 \mathcal E_{m,x}(\tau,w):=
 \sum_{n\in\ZZ}x^{-2n}q^{mn^2}u^{2mn}
       \mathcal G_{d,x}(q^nu;q).
\end{equation}
We let
\begin{equation*}%\label{eq:Dx}
 \mathcal D_x:=\{(\tau,w)\in\HH\times\CC:
           w=-\alpha+a\tau+b,\ a,b\in\ZZ\}
\end{equation*}
and define
\begin{equation}\label{eq:rho-def}
 \rho_{m,x}(\tau):=x^{-2}\sum_{k\ge0}
          x^{2(m+1)k}q^{(m+1)k^2}g_{k,x}(q).
\end{equation}

Our second theorem describes the analytic properties of this periodization: its convergence, its poles and residues, and its elliptic transformation law.

\begin{theorem}\label{thm:elliptic}
\begin{enumerate}
    \item[(i)] The series $\mathcal E_{m,x}(\tau,w)$ converges normally on compact subsets of
$(\HH\times\CC)\setminus\mathcal D_x$ and extends jointly
meromorphically, with at most simple poles on $\mathcal D_x$ and no
other poles. For all $a,b\in\ZZ$, we have
\begin{equation}\label{eq:twisted-law}
 \mathcal E_{m,x}(\tau,w+a\tau+b)
 =x^{2a}q^{-ma^2}u^{-2ma}\mathcal E_{m,x}(\tau,w).
\end{equation}

\item[(ii)]
The function $\rho_{m,x}$ is holomorphic on $\HH$, and
\begin{equation*}%\label{eq:E-residue}
 \Res_{w=-\alpha+a\tau+b}\mathcal E_{m,x}(\tau,w)
 =-\frac{\rho_{m,x}(\tau)}{2\pi i}
        x^{2(m+1)a}q^{-ma^2}.
\end{equation*}
Moreover, we have
\begin{equation}\label{eq:rho-leading}
 \rho_{m,x}(\tau)=x^{-2}(1+O(q))\qquad(q\to0),
\end{equation}
so $\mathcal E_{m,x}$ is nontrivial.

\item[(iii)] Fix $\tau \in \mathbb H$. Then the poles of $\mathcal E_{m,x}(\tau,w)$ are simple on the entire orbit if $\rho_{m,x}(\tau)\ne0$ and are
all removable otherwise. 

\item[(iv)] 
The translate
\begin{equation*}%\label{eq:E-translation}
 \widetilde{\mathcal E}_{m,x}(\tau,w):=
    \mathcal E_{m,x}(\tau,w+\alpha/m)
\end{equation*}
satisfies the untwisted law
\begin{equation}\label{eq:untwisted-law}
 \widetilde{\mathcal E}_{m,x}(\tau,w+a\tau+b)
 =q^{-ma^2}u^{-2ma}\widetilde{\mathcal E}_{m,x}(\tau,w).
\end{equation}

\end{enumerate}
\end{theorem}

To remove the poles, we use the kernel
\begin{equation}\label{eq:K-def}
 \mathcal K_{m,x}(\tau,w):=
 \sum_{n\in\ZZ}\frac{x^{-2n}q^{mn^2}u^{2mn}}{1-xuq^n},
\end{equation}
and define the finite part and its translate by
\begin{equation}\label{eq:finite-def}
 \mathcal E^{\mathrm{fin}}_{m,x}:=
       \mathcal E_{m,x}-\rho_{m,x}\mathcal K_{m,x},
 \qquad
 \widetilde{\mathcal E}^{\mathrm{fin}}_{m,x}(\tau,w):=
       \mathcal E^{\mathrm{fin}}_{m,x}(\tau,w+\alpha/m).
\end{equation}
Finally, for $0\le\ell<2m$, we define the theta functions
\begin{equation*}%\label{eq:theta-def}
 \vartheta_{m,\ell}(\tau,w):=
 \sum_{\substack{r\in\ZZ\\r\equiv\ell\ ({\rm mod}\ 2m)}}
            q^{r^2/(4m)}u^r.
\end{equation*}

Our third theorem shows that the finite part is holomorphic, and it gives its theta decomposition explicitly.

\begin{theorem}\label{thm:theta}
The function $\mathcal E^{\mathrm{fin}}_{m,x}$ is jointly holomorphic on
$\HH\times\CC$ and satisfies the elliptic transformation law as in \eqref{eq:twisted-law}. Its translate
satisfies the untwisted elliptic transformation law in \eqref{eq:untwisted-law} and has the unique expansion
\begin{equation}\label{eq:theta-expansion}
 \widetilde{\mathcal E}^{\mathrm{fin}}_{m,x}(\tau,w)
 =\sum_{\ell=0}^{2m-1}h_{m,\ell,x}(\tau)\vartheta_{m,\ell}(\tau,w).
\end{equation}
The coefficients are holomorphic on $\HH$ and are given by
\begin{equation*}%\label{eq:h-explicit}
 \begin{split}
 h_{m,\ell,x}(\tau)
 ={}&e^{2\pi i\ell\alpha/m}x^{\ell-2}q^{-\ell^2/(4m)}\\
 &\times\sum_{k\ge1}g_{k,x}(q)
       \sum_{s=1}^{k}x^{2(m+1)(k-s)}
           q^{(m+1)k^2-ms^2+\ell s}.
 \end{split}
\end{equation*}
These series converge normally on compact subsets of $\HH$. As $q\to0$, we have
\begin{equation}\label{eq:h-leading}
 h_{m,\ell,x}(\tau)
 =e^{2\pi i\ell\alpha/m}x^{\ell-2}
       q^{1+\ell-\ell^2/(4m)}(1+O(q^2)).
\end{equation}
For $x=1$, all the coefficient functions are strictly
positive at $\tau=it$, $t>0$.
\end{theorem}

The elliptic law and theta-decomposition mechanism in
Theorems~\ref{thm:elliptic} and~\ref{thm:theta} extend to
Lambert-type seeds satisfying suitable convergence hypotheses.
Their nontriviality, asymptotic, and positivity conclusions
require additional information about the coefficients.
The finite recurrence of Theorem~\ref{thm:finite} gives a further
characterization within the following specific family.
We consider sequences $Y=(Y_j)_{j\ge0}$ satisfying
\eqref{eq:recurrence} for $j\ge2$, with $X_j^{(m)}(x;q)$
replaced by $Y_j$ and $G_d(x;q)$ by a scalar $\gamma$.

\begin{theorem}
\label{rem:recurrence-implies}
Fix $m\ge2$, $d=2m+1$, a root of unity $x$, and $0<|q|<1$.

\begin{enumerate}[label=\textup{(\roman*)}]
\item
The unique sequence in this family tending to $1$ is
$Y_j=X_j^{(m)}(x;q)/G_d(x;q)$.
This includes $C_x=0$, when $\gamma$ has no effect on the recurrence.

\item
The seed
\[
 \mathcal G_Y(u)=x^{-2}\sum_{j\ge0}Y_j(xu)^j
\]
is a series of Lambert type; it admits meromorphic continuation, and its bilateral periodization, formed as in
\eqref{eq:E-def}, converges and satisfies \eqref{eq:twisted-law}.
Outside a discrete exceptional subset of $0<|q|<1$, its poles
are contained in the prescribed orbit $\mathcal D_x$ if and only
if $Y$ is a scalar multiple of $X^{(m)}(x;q)$.
The exceptional set is disjoint from $0<|q|\le1/2$.
This characterizes only the location of the poles:
non-Dyson solutions can also have a single orbit of simple poles.

\item
At $x=1$, impose $\gamma=1$ and terminal values
$Y_J=\cdots=Y_{J+2m-1}=1$, and solve backwards; the
$2m=d-1$ prescribed values are exactly the data consumed by one
backward step.
For every integer $J\ge0$, the resulting initial value is a
polynomial $P_J\in\ZZ[q]$ satisfying
\[
 \RRm_m(\zeta_d;q)\equiv P_J(q)\pmod{q^{J+1}}.
\]
Thus the normalized recurrence determines the primitive rank
specialization.
\end{enumerate}

\end{theorem}

%The recurrence and the elliptic completion use different features of
%the same seed. Removing the first factor of $G_d$ gives a polynomial
%of degree $d-1$. Substitution in the defining $k$-sum and the change
%of index $k\mapsto k+1$ produce the finite recurrence. Summing that
%recurrence in the auxiliary variable gives the polynomial defect in
%Theorem~\ref{thm:finite}. 

%For the analytic construction, the useful expression is instead the
%simple-fraction expansion \eqref{eq:seed-fractions}. The bilateral
%average places its poles on a full lattice orbit. Its elliptic law
%comes from reindexing this average. The Gaussian factors make both
%tails converge normally, including when $\tau$ varies, and allow the
%%residues to be computed term by term. This gives the residue function
%$\rho_{m,x}$ and identifies the multiple of $\mathcal K_{m,x}$ that
%removes all poles.

%The theta-coefficient calculation can then be reduced to finite
%algebra. After reindexing the $k$th contribution to the average, its
%polar subtraction is governed by the following identity for an integer
%$L\ge0$ and $v\in\CC\setminus\{0,1\}$:
%\[
% \frac{v^{-L}-1}{1-v}=\sum_{j=1}^{L}v^{-j}\qquad(L\ge0),
%\]
%with the empty-sum convention for $L=0$. In the present setting
%$L=2mk$. Sorting these finitely many powers into classes modulo $2m$
%produces the theta basis and the inner sum over $1\le s\le k$ in
%\eqref{eq:h-explicit}. The first contribution, $k=s=1$, gives the
%nonzero leading term. The finite theta part is therefore established
%by an explicit identity, not only by a dimension argument.

The paper is organized as follows.
Section~\ref{sec:dyson} develops the combinatorial objects and their
generating identities. Section~\ref{sec:recurrence} proves Theorem~\ref{thm:finite}, the finite
recurrence and defect formula. Section~\ref{sec:preliminaries} supplies
the elliptic preliminaries, including full proofs of theta
decomposition and local polar subtraction.
In Sections~\ref{sec:elliptic} and~\ref{sec:theta} we prove Theorems~\ref{thm:elliptic} and~\ref{thm:theta}.
Section~\ref{sec:remark-proof} proves the three claims of Theorem~\ref{rem:recurrence-implies}, which characterize the Dyson seed among all solutions of the finite recurrence (Theorem~\ref{thm:remark-proof}). Finally, in Section~\ref{sec:example} we work out the first new case $m=2$, both at the primitive specialization and at the other fifth-root deformations, and we conclude with a description of the Lean formalization.

\section*{Acknowledgements}
\noindent The authors thank George Andrews for discussions related to this work,
and Kenny Lau for organizing material related to AxiomProver.

\section{Combinatorial interpretation of the higher Dyson series}\label{sec:dyson}

In this section, we describe the combinatorial source of the higher Dyson series and prove the collapse at roots of unity that was used in the introduction.

We begin by recalling
\[
\mathcal{R}_m(x_1,\dots,x_m;q):=\sum_{n\ge 0}\frac{q^{n^2}}{\prod_{j=1}^m (x_jq,q/x_j;q)_n},
\]
and its one-parameter specialization
\[
\mathcal{R}_m(z;q):=\mathcal{R}_m(z,z^2,\dots,z^m;q)=\sum_{n\ge 0}\frac{q^{n^2}}{\prod_{j=1}^m (z^jq,z^{-j}q;q)_n}.
\]
For $m=1$ this is Dyson's rank generating function. If $\zeta_d$ is a primitive $d$-th root of unity, then for every $k\geq 1$ we have
\[
\prod_{j=1}^m (1-\zeta_d^jq^k)(1-\zeta_d^{-j}q^k)= \prod_{r=1}^{d-1}(1-\zeta_d^rq^k)=\frac{1-q^{dk}}{1-q^k},
\]
since $\{\pm j\,:\, 1\leq j\leq m\}$ is a complete set of nonzero residues modulo $d=2m+1$. This implies
\[
\mathcal{R}_m(\zeta_d; q)= \sum_{n\geq 0} q^{n^2} \frac{(q;q)_n}{(q^d;q^d)_n}.
\]

We now define uncoupled $m$-Dyson symbols, the combinatorial objects whose generating series is $\mathcal{R}_m(z;q)$. These are natural companions of the \emph{coupled} Dyson symbols considered by Andrews in \cite{andrews}.

\begin{definition}\label{def:dyson-symbol}
An \emph{uncoupled $m$-Dyson symbol} is given by data
\[
\Sigma=(n;\alpha_1,\beta_1,\dots,\alpha_m,\beta_m),
\]
where $n\ge 0$ and, for each $j$, the partitions $\alpha_j$ and $\beta_j$ have largest part at most $n$. We define the \emph{size} and \emph{full rank of $\Sigma$} by
\[
|\Sigma|:=n^2+\sum_{j=1}^m\bigl(|\alpha_j|+|\beta_j|\bigr),
\]
\[
\operatorname{FR}_m(\Sigma):=\sum_{j=1}^m j\bigl(\ell(\alpha_j)-\ell(\beta_j)\bigr).
\]
Here, $\ell(\lambda)$ denotes the length of the partition $\lambda$, i.e., the number of its parts.
\end{definition}

\begin{example}
We give an example of an \emph{uncoupled $2$-Dyson symbol}.
Take $n=3$ and let
\[
\alpha_1 =(3,1),\, \beta_1=(2,2),\, \alpha_2=(1),\,\beta_2=\emptyset.
\]
Then we have
\[
|\alpha_1| = 4, \, \ell(\alpha_1) = 2, \quad |\beta_1| = 4, \, \ell(\beta_1) = 2, \quad |\alpha_2| = 1, \, \ell(\alpha_2) = 1, \quad |\beta_2| = 0, \, \ell(\beta_2) = 0.
\]
The resulting uncoupled $2$-Dyson symbol is 
\[
\Sigma = \bigl(3; (3,1), (2,2), (1), \emptyset\bigr).
\]
Its size and full rank are
\begin{align*}
|\Sigma| &:= n^2 + \sum_{j=1}^2 \bigl(|\alpha_j| + |\beta_j|\bigr) =18, \quad \operatorname{FR}_2(\Sigma) := \sum_{j=1}^2 j\bigl(\ell(\alpha_j) - \ell(\beta_j)\bigr) =2.
\end{align*}
\end{example}

We now show that $\mathcal{R}_m(z;q)$ is the generating function of uncoupled $m$-Dyson symbols.
\begin{theorem}\label{thm:combinatorialgenseries}
We have
\[
\mathcal{R}_m(z;q)=\sum_{\Sigma} z^{\operatorname{FR}_m(\Sigma)}q^{|\Sigma|},
\]
where the sum runs over all uncoupled $m$-Dyson symbols. Equivalently,
\[
\mathcal{R}_m(z;q)=\sum_{N\ge 0}\sum_{r\in\ZZ} N_m(r,N)z^rq^N,
\]
where $N_m(r,N)$ counts uncoupled $m$-Dyson symbols of size $N$ and full rank $r$. In particular,
\[
N_m(r,N)=N_m(-r,N).
\]
\end{theorem}
\begin{proof}
Fix $n\ge 0$. For each $j$,
\[
\frac{1}{(z^jq;q)_n}=\sum_{\alpha_j} z^{j\ell(\alpha_j)}q^{|\alpha_j|},
\qquad
\frac{1}{(z^{-j}q;q)_n}=\sum_{\beta_j} z^{-j\ell(\beta_j)}q^{|\beta_j|},
\]
where $\alpha_j$ and $\beta_j$ run over partitions whose largest parts are at most $n$. Multiplying these identities for $j=1,\dots,m$ and then multiplying by $q^{n^2}$ yields
\[
\frac{q^{n^2}}{\prod_{j=1}^m (z^jq,z^{-j}q;q)_n}
=
\sum_{\alpha_1,\beta_1,\dots,\alpha_m,\beta_m}
 z^{\sum_{j=1}^m j(\ell(\alpha_j)-\ell(\beta_j))}
 q^{n^2+\sum_{j=1}^m(|\alpha_j|+|\beta_j|)}.
\]
Summing over $n\ge 0$ gives the result.

The involution
\[
(n;\alpha_1,\beta_1,\dots,\alpha_m,\beta_m)
\longmapsto
(n;\beta_1,\alpha_1,\dots,\beta_m,\alpha_m)
\]
preserves the size and negates the full rank. Therefore, for every $m\ge 1$, every $r\in\ZZ$, and every $N\ge 0$, we have
\[
N_m(r,N)=N_m(-r,N).
\]
\end{proof}

For an indeterminate $x$, we let $\Theta(x;q):=(xq,q/x;q)_\infty$.
For $s\in\ZZ_{\ge 0}$ and $n\ge 0$, we define
\[
A_s(x;n):=(xq^{n+1+s},q^{n+1}/x;q)_\infty.
\]
If $\mathbf s=(s_1,\dots,s_m)\in\ZZ_{\ge 0}^m$, let
$
|\mathbf s|:=s_1+\cdots+s_m,
$
and set
\[
\FF_m^{\mathbf s}(x_1,\dots,x_m;q):=\sum_{n\ge 0} q^{n^2+|\mathbf s|n}\prod_{j=1}^m A_{s_j}(x_j;n).
\]
Furthermore, we let \(e_i\) denote the \(i\)-th standard basis vector of \(\mathbb Z^m\).
The next theorem expresses the product of $\mathcal{R}_m$ with the infinite products $\Theta(x_j;q)$ as the series $\FF_m^{\mathbf 0}$, and it gives a shift identity for the whole family $\FF_m^{\mathbf s}$. By Jacobi's triple product identity, $(1-x)(q;q)_\infty\Theta(x;q)=\sum_{n\in\ZZ}(-1)^nq^{n(n-1)/2}x^n$ is a theta function; $\Theta(x;q)$ itself satisfies $\Theta(qx;q)/\Theta(x;q)=(1-x^{-1})/(1-xq)$.

\begin{theorem}\label{thm:multivariable-shift}
  We have
\[
\Bigl(\prod_{j=1}^m \Theta(x_j;q)\Bigr)\mathcal{R}_m(x_1,\dots,x_m;q)=\FF_m^{\mathbf 0}(x_1,\dots,x_m;q),
\]
and for $1\le i\le m$ and every $\mathbf s\in\ZZ_{\ge 0}^m$,
\[
\begin{aligned}
\FF_m^{\mathbf s}(x_1,\dots,qx_i,\dots,x_m;q)&-\FF_m^{\mathbf s}(x_1,\dots,x_i,\dots,x_m;q) \\
&=\bigl(q^{s_i+1}x_i-x_i^{-1}\bigr)\FF_m^{\mathbf s+\mathbf e_i}(x_1,\dots,x_m;q).
\end{aligned}
\]
\end{theorem}
\begin{proof}
      For each fixed $j$ and each $n\ge 0$, we have
\[
\frac{\Theta(x_j;q)}{(x_jq,q/x_j;q)_n}
=
\frac{(x_jq,q/x_j;q)_\infty}{(x_jq,q/x_j;q)_n}
=(x_jq^{n+1},q^{n+1}/x_j;q)_\infty
=A_0(x_j;n).
\]
Multiplying over $j=1,\dots,m$ and summing over $n\ge 0$ gives for every $m\ge 1$,
\[
\Bigl(\prod_{j=1}^m \Theta(x_j;q)\Bigr)\mathcal{R}_m(x_1,\dots,x_m;q)=\FF_m^{\mathbf 0}(x_1,\dots,x_m;q).
\]
Moreover, we have
\[
A_s(qx;n)=(1-q^n/x)(xq^{n+2+s},q^{n+1}/x;q)_\infty,
\]
\[
A_s(x;n)=(1-xq^{n+1+s})(xq^{n+2+s},q^{n+1}/x;q)_\infty.
\]
Subtracting gives
\[
A_s(qx;n)-A_s(x;n)=q^n\bigl(q^{s+1}x-x^{-1}\bigr)A_{s+1}(x;n).
\]
We apply this to the $i$th factor in the definition of $\FF_m^{\mathbf s}$ and obtain
\begin{align*}
\FF_m^{\mathbf s}&(x_1,\dots,qx_i,\dots,x_m;q)-\FF_m^{\mathbf s}(x_1,\dots,x_i,\dots,x_m;q)
\\
&=\sum_{n\ge 0} q^{n^2+|\mathbf s|n}
\Bigl(\prod_{j\neq i}A_{s_j}(x_j;n)\Bigr)
q^n\bigl(q^{s_i+1}x_i-x_i^{-1}\bigr)A_{s_i+1}(x_i;n).
\end{align*}
We factor out $q^{s_i+1}x_i-x_i^{-1}$. Since $|\mathbf s+\mathbf e_i|=|\mathbf s|+1$, the remaining sum is precisely
\[
\FF_m^{\mathbf s+\mathbf e_i}(x_1,\dots,x_m;q),
\]
which implies the result.
\end{proof}

Before stating the next theorem, we recall our notation for Gaussian coefficients. For integers
\(A\) and \(B\), define
\[
\begin{bmatrix} A \\ B \end{bmatrix}_q
:=
\begin{cases}
\dfrac{(q;q)_A}{(q;q)_B(q;q)_{A-B}}, & 0 \leq B \leq A,\\[1.2em]
0, & \text{otherwise}.
\end{cases}
\]
The finite \(q\)-binomial theorem then gives, for \(n\geq 1\),
\[
\frac{1}{(xq;q)_n}
=
\sum_{a\geq 0}
\begin{bmatrix} n+a-1 \\ a \end{bmatrix}_q
x^a q^a.
\]
This is the convention used in Theorem~\ref{thm:gaussian-coefficients}.

\begin{theorem}\label{thm:gaussian-coefficients}
    For each $n\ge 0$, define coefficients $C_n(k;q)$ by
\[
\frac{1}{(xq,q/x;q)_n}=\sum_{k\in\ZZ} C_n(k;q)x^k.
\]
For $n=0$ one has $C_0(0;q)=1$ and $C_0(k;q)=0$ for $k\neq 0$. For every $n\ge 1$ and every $k\ge 0$, and for $|q|<|x|<|q|^{-1}$,
\[
C_n(k;q)=\sum_{v\ge 0}
\qbinom{n+k+v-1}{k+v}
\qbinom{n+v-1}{v}
q^{k+2v},
\qquad
C_n(-k;q)=C_n(k;q),
\]
and
\[
\mathcal{R}_m(z;q)
=
\sum_{n\ge 0} q^{n^2}
\sum_{\mathbf k\in\ZZ^m}
\Bigl(\prod_{j=1}^m C_n(k_j;q)\Bigr)
 z^{k_1+2k_2+\cdots+mk_m},
\]
which holds for $|q|^{1/m}<|z|<|q|^{-1/m}$.
\end{theorem}
\begin{proof}
    For $n=0$ the definition gives
$
\frac{1}{(xq,q/x;q)_0}=1,
$
so $C_0(0;q)=1$ and $C_0(k;q)=0$ for $k\neq 0$. Now assume $n\ge 1$. By the finite $q$-binomial theorem,
\[
\frac{1}{(xq;q)_n}=\sum_{a\ge 0}\qbinom{n+a-1}{a}x^a q^a,
\qquad
\frac{1}{(q/x;q)_n}=\sum_{b\ge 0}\qbinom{n+b-1}{b}x^{-b} q^b.
\]
Multiplying gives
\[
\frac{1}{(xq,q/x;q)_n}
=
\sum_{a,b\ge 0}
\qbinom{n+a-1}{a}
\qbinom{n+b-1}{b}
 x^{a-b} q^{a+b}.
\]
For $k\ge 0$, the coefficient of $x^k$ comes from $a=b+k$. Writing $v:=b$, we obtain
\[
C_n(k;q)=\sum_{v\ge 0}
\qbinom{n+k+v-1}{k+v}
\qbinom{n+v-1}{v}
q^{k+2v}.
\]
The symmetry $C_n(-k;q)=C_n(k;q)$ is immediate by interchanging $a$ and $b$.
Now, for each fixed $n$ and $j$,
\[
\frac{1}{(z^jq,z^{-j}q;q)_n}=\sum_{k_j\in\ZZ} C_n(k_j;q)z^{jk_j}.
\]
Multiplying over $j=1,\dots,m$, then multiplying by $q^{n^2}$, then summing over $n\ge 0$, gives the desired expansion.
\end{proof}

We now turn to the distribution of the full rank modulo $d$.
For $N\ge 0$ and $a\in\ZZ/d\ZZ$, let $N_m(a,d;N)$ denote the number of uncoupled $m$-Dyson symbols of size $N$ whose full rank is congruent to $a\pmod d$. The next theorem shows that, when $d$ is prime, the nonzero residue classes are equally populated.

\begin{theorem}\label{thm:appendixcongruence}
Fix $m\ge 1$.
If $d=2m+1$ is prime, then for every $N\ge 0$ one has
\[
N_m(1,d;N)=N_m(2,d;N)=\cdots=N_m(d-1,d;N).
\]
Equivalently, there exist integers $A_m(N)$ and $B_m(N)$ such that
\[
N_m(0,d;N)=A_m(N),
\qquad
N_m(a,d;N)=B_m(N)\quad (a=1,\dots,d-1),
\]
and hence
\[
[q^N]\mathcal{R}_m(\zeta_d;q)=A_m(N)-B_m(N),
\]
where $[q^N]$ is the operator that extracts the $N$-th coefficient.
\end{theorem}

\begin{proof}
Fix $N \ge 0$ and define $F_N(X) := \sum_{a=0}^{d-1} N_m(a,d;N) X^a \in \mathbb{Z}[X]$. By definition, $[q^N]\mathcal{R}_m(\zeta_d^r;q) = F_N(\zeta_d^r)$. Since $\mathcal{R}_m(\zeta_d^r;q)$ is independent of $r$ for all $r \in (\mathbb{Z}/d\mathbb{Z})^\times$, the value $F_N(\zeta_d)$ is fixed by the Galois group $\text{Gal}(\mathbb{Q}(\zeta_d)/\mathbb{Q})$, hence rational. Since it is an algebraic integer, it is an integer. Thus, $F_N(\zeta_d) = \delta_N$ for some $\delta_N \in \mathbb{Z}$.

The polynomial $G_N(X) := F_N(X) - \delta_N$ vanishes at all primitive $d$-th roots of unity, and is therefore divisible by the cyclotomic polynomial $\Phi_d(X)$. Since $d$ is prime, $\Phi_d(X) = \sum_{a=0}^{d-1} X^a$. Because $\deg(G_N) \le d-1$ and $\deg(\Phi_d) = d-1$, there must exist a constant $c_N \in \mathbb{Z}$ such that
\[
F_N(X) - \delta_N = c_N \Phi_d(X) = c_N \sum_{a=0}^{d-1} X^a.
\]
Comparing coefficients of $X^a$ yields
\begin{align*}
N_m(0,d;N) - \delta_N &= c_N, \\
N_m(a,d;N) &= c_N \quad (1 \le a \le d-1).
\end{align*}
Setting $B_m(N) := c_N$ and $A_m(N) := \delta_N + c_N$, we obtain $N_m(a,d;N) = B_m(N)$ for $a \ne 0$ and $N_m(0,d;N) = A_m(N)$. Finally, the specialized coefficient is given by
\[
[q^N]\mathcal{R}_m(\zeta_d;q) = \delta_N = A_m(N) - B_m(N),
\]
completing the proof.
\end{proof}

\begin{corollary}\label{cor:recover}
Let
\[
U_m(q):=\RRm_m(1;q)=\sum_{n\ge 0}\frac{q^{n^2}}{(q;q)_n^{2m}}=\sum_{N\ge 0}U_m(N)q^N.
\]
If $d=2m+1$ is prime and
\[
D_m(N):=[q^N]\RRm_m(\zeta_d;q),
\]
then
\[
A_m(N)=\frac{U_m(N)+(d-1)D_m(N)}{d},
\qquad
B_m(N)=\frac{U_m(N)-D_m(N)}{d}.
\]
\end{corollary}

\begin{proof}
The total number of uncoupled $m$-Dyson symbols of size $N$ is
\[
U_m(N)=A_m(N)+(d-1)B_m(N),
\]
while Theorem~\ref{thm:appendixcongruence} gives
\[
D_m(N)=A_m(N)-B_m(N).
\]
Solving this $2\times 2$ linear system gives the result.
\end{proof}

\section{The higher Dyson deformation and its recurrence}
\label{sec:recurrence}

Recall $q=e^{2\pi i\tau}$ with $\tau\in\HH$, so $|q|<1$.
The main analytic argument uses the deformed sequence \(X_j^{(m)}(x;q)\)
rather than the full combinatorial model. By \eqref{eq:intro-normalized-root}, the specialization \(x=1\),
\(j=0\), recovers the normalized primitive \(d\)-th-root specialization of
\(\RRm_m\).

Let $m\geq 2$ and set $d:=2m+1$. From the introduction, we recall that
\[
G_d(x;q):=\frac{(x^dq^d;q^d)_\infty}{(xq;q)_\infty}.
\]
For every integer $j$, we let
\[
X_j^{(m)}(x;q):=\sum_{n\geq 0} q^{n^2+jn}G_d(xq^n;q).
\]
The corresponding generating function in the variable $t$ is
\[
\mathcal F_{d,x}(t;q):=\sum_{j\geq 0} X_j^{(m)}(x;q)t^j.
\]
The next lemma records the basic meromorphic structure of $\mathcal F_{d,x}(t;q)$ in the variable $t$.

\begin{lemma}\label{lem:seed}
Let $x\in\CC^\times$ and $|q|<1$. Assume that
$x\notin q^{-\ZZ_{\geq 1}}$, so that the denominators $(xq^{n+1};q)_\infty$ of the functions $G_d(xq^n;q)$ do not vanish. This condition holds, in
particular, for every root of unity $x$. Then one has
\[
\mathcal F_{d,x}(t;q)
=
\sum_{n\geq 0}\frac{q^{n^2}G_d(xq^n;q)}{1-tq^n}
\qquad (|t|<1).
\]
Consequently $\mathcal F_{d,x}(t;q)$ extends meromorphically to the $t$-plane, with at most simple poles at the points $t=q^{-n}$ $(n\geq 0)$.
\end{lemma}

\begin{proof}
For $|t|<1$ the defining series converges absolutely, so one may interchange the $j$-sum and the $n$-sum:
\[
\mathcal F_{d,x}(t;q)
=
\sum_{j\geq 0}\sum_{n\geq 0} q^{n^2+jn}G_d(xq^n;q)t^j
=
\sum_{n\geq 0} q^{n^2}G_d(xq^n;q)\sum_{j\geq 0}(tq^n)^j.
\]
Summing the geometric series gives the stated formula. The right-hand side is meromorphic in $t$ with only the indicated simple poles.
\end{proof}

The next step is to extract a finite \(q\)-difference relation from the higher
Dyson series. This recurrence is the algebraic engine of the paper: at nontrivial
\(d\)-th roots of unity it collapses to a homogeneous relation, while away from
those roots it produces the explicit defect $C_x G_d(x;q)$.

The recurrence follows from the elementary multiplication identity
\[
G_d(x;q)=\bigl(1+xq+x^2q^2+\cdots+x^{d-1}q^{d-1}\bigr)G_d(xq;q)
\qquad(x\notin q^{-\ZZ_{\geq1}}),
\]
which is obtained from \eqref{eq:G-product} by writing
$(x^dq^d;q^d)_\infty=(1-x^dq^d)(x^dq^{2d};q^d)_\infty$,
$(xq;q)_\infty=(1-xq)(xq^2;q)_\infty$ and
$(1-x^dq^d)/(1-xq)=\sum_{r=0}^{d-1}(xq)^r$. We do not divide by
$G_d(xq;q)$, which can vanish under our hypotheses (for instance for
$x=\omega q^{-2}$ with $\omega\neq1$ a $d$-th root of unity).

\begin{proof}[Proof of Theorem~\ref{thm:finite}]
Applying the multiplication identity above with $x$ replaced by $xq^n$
(note that $xq^n\notin q^{-\ZZ_{\geq1}}$ for $n\geq0$), we obtain
\[
G_d(xq^n;q)=\Bigl(\sum_{r=0}^{d-1}x^rq^{r(n+1)}\Bigr)G_d(xq^{n+1};q).
\]
We multiply by $q^{n^2+jn}$ and sum over $n\geq 0$, which gives
\[
X_j^{(m)}(x;q)
=
\sum_{r=0}^{d-1}x^r\sum_{n\geq 0}q^{n^2+(j+r)n+r}G_d(xq^{n+1};q).
\]
Setting $N:=n+1$ yields
\[
n^2+(j+r)n+r=N^2+(j+r-2)N+(1-j).
\]
Therefore, we have
\[
X_j^{(m)}(x;q)
=
q^{1-j}\sum_{r=0}^{d-1}x^r\sum_{N\geq 1} q^{N^2+(j+r-2)N}G_d(xq^N;q).
\]
The $N=0$ term of $X_{j+r-2}^{(m)}(x;q)$ is precisely $G_d(x;q)$. Hence
\[
X_j^{(m)}(x;q)
=
q^{1-j}\sum_{r=0}^{d-1}x^r\bigl(X_{j+r-2}^{(m)}(x;q)-G_d(x;q)\bigr),
\]
which is equivalent to \eqref{eq:recurrence}.

To sum \eqref{eq:recurrence} over $j\ge0$, we first express the terms $X_{-1}$ and
$X_{-2}$ using nonnegative indices. Write $g:=G_d(x;q)$.
The cases $j=1$ and $j=0$ are
\[
 \begin{aligned}
 X_{-1}+\sum_{r=1}^{d-1}x^rX_{r-1}-X_1&=C_xg,\\
 X_{-2}+xX_{-1}+\sum_{r=2}^{d-1}x^rX_{r-2}-q^{-1}X_0&=C_xg.
 \end{aligned}
\]
Solving the first and substituting into the second gives the boundary values
\begin{align}
 X_{-1}&=C_xg-\sum_{r=1}^{d-1}x^rX_{r-1}+X_1,
       \label{eq:boundary1}\\
 X_{-2}&=q^{-1}X_0-xX_1+x^dX_{d-2}+(1-x)C_xg.
       \label{eq:boundary2}
\end{align}
Let $F(t)=\mathcal F_{d,x}(t;q)$. For $r=0,1$, we have
\[
 \sum_{j\ge0}X_{j-2}t^j=X_{-2}+tX_{-1}+t^2F(t),\qquad
 \sum_{j\ge0}X_{j-1}t^j=X_{-1}+tF(t).
\]
For $2\le r\le d-1$, we have
\[
 \sum_{j\ge0}X_{j+r-2}t^j
 =t^{2-r}\left(F(t)-\sum_{k=0}^{r-3}X_kt^k\right).
\]
Therefore
\begin{equation}\label{eq:F-intermediate}
 q^{-1}F(qt)=\sum_{r=0}^{d-1}x^rt^{2-r}F(t)+B(t)-\frac{C_xg}{1-t},
\end{equation}
with
\[
 B(t):=X_{-2}+(x+t)X_{-1}
       -\sum_{r=2}^{d-1}x^r\sum_{k=0}^{r-3}X_kt^{k+2-r}.
\]
These identities initially hold on $0<|t|<1$. By Lemma~\ref{lem:seed} the infinite sums converge.

We set $t=xu$ in \eqref{eq:F-intermediate}, divide by $x^2$, and substitute
 \eqref{eq:boundary1} and \eqref{eq:boundary2} to obtain
\[
 u^{d-3}x^{-2}\{B(xu)-C_xg(1+xu)\}=R_{d,x}(u;q).
\]
Collecting coefficients gives \eqref{eq:R-explicit}.

The identities
\[
 \sum_{r=0}^{d-1}u^{2-r}=\frac{u^d-1}{u^{d-3}(u-1)},\qquad
 1+xu-\frac1{1-xu}=-\frac{x^2u^2}{1-xu}
\]
then yield \eqref{eq:seed-difference}.
The explicit formula gives the degree bound and the constant term, and only the term $q^{-1}X_0$ can contribute a pole at $q=0$.
Finally, if $R'$ is another polynomial satisfying \eqref{eq:seed-difference}, then subtracting the two functional equations gives $qu^{3-d}(R-R')=0$ on an open set,
which proves uniqueness.
\end{proof}

Finally, we note that the finite data in Theorem~\ref{thm:finite} also determine a rational
partial-fraction expansion. For this we first remove the restriction
$x\notin q^{-\ZZ_{\geq1}}$. By \eqref{eq:G-product}, for $yq^r\neq1$
($r\geq1$) we have
\begin{equation}\label{eq:G-product-form}
 G_d(y;q)=\prod_{r\ge1}\frac{1-y^dq^{dr}}{1-yq^r}
 =\prod_{r\ge1}\bigl(1+yq^r+\cdots+(yq^r)^{d-1}\bigr),
\end{equation}
and the product on the right converges normally on compact subsets of
$\CC\times\{|q|<1\}$. We use \eqref{eq:G-product-form} to extend
$G_d(y;q)$ to a function holomorphic in $(y,q)$ on $\CC\times\{|q|<1\}$;
this removes the excluded denominator zeros, and the multiplication
identity $G_d(y;q)=\bigl(\sum_{r=0}^{d-1}(yq)^r\bigr)G_d(yq;q)$ then holds
for all $y\in\CC$ by continuity. Consequently $X_j^{(m)}(x;q)$, the
recurrence \eqref{eq:recurrence}, the functional equation
\eqref{eq:seed-difference} and the polynomial $R_{d,x}(u;q)$ make sense
for every $x\in\CC^\times$, and $X_j^{(m)}(x;q)$ and the coefficients of
$R_{d,x}(u;q)$ are holomorphic functions of $q$ on $0<|q|<1$ for every
fixed $x\in\CC^\times$ (with a simple pole at $q=0$ at most).

To this end, for arbitrary $x\in\CC^\times$, we normalize
the rational inhomogeneous term of \eqref{eq:seed-difference} by putting
\[
 f_{d,x}(u;q):=\frac{1-u}{u(1-u^d)}
 \left(R_{d,x}(u;q)-\frac{C_xu^{d-1}}{1-xu}G_d(x;q)\right).
\]
 Put
\[
 \varepsilon_x:=\begin{cases}1,&x^d\ne1,\\0,&x^d=1,\end{cases}
\]
and, writing $R=R_{d,x}(u;q)$ and $g=G_d(x;q)$, define
\begin{equation*}%\label{eq:Q-def}
 Q_{d,x}(u;q):=
 \begin{cases}
 (1-xu)(1-u)R-C_xu^{d-1}(1-u)g,&x^d\ne1,\\
 (1-u)R-du^{d-1}g,&x=1,\\
 (1-u)R,&x^d=1,\ x\ne1.
 \end{cases}
\end{equation*}
Then $f_{d,x}=Q_{d,x}/\{u(1-u^d)(1-xu)^{\varepsilon_x}\}$.

\begin{proposition}\label{prop:partial-fractions}
Let $\omega=e^{2\pi i/d}$. There are unique coefficients $D_{r,x}(q)$
for $0\le r<d$, and $E_x(q)$ when $x^d\ne1$, such that
\begin{equation}\label{eq:partial-fractions}
 \frac{Q_{d,x}(u;q)}{u(1-u^d)(1-xu)^{\varepsilon_x}}
 =-\frac{X_0^{(m)}(x;q)}u
       +\sum_{r=0}^{d-1}\frac{D_{r,x}(q)}{1-\omega^ru}
       +\varepsilon_x\frac{E_x(q)}{1-xu}.
\end{equation}
When $x^d\ne1$, they are explicitly given by 
\[
 D_{r,x}(q)=\frac{\omega^rQ_{d,x}(\omega^{-r};q)}
                       {d(1-x\omega^{-r})},\qquad
 E_x(q)=\frac{xQ_{d,x}(x^{-1};q)}{1-x^{-d}}.
\]
When $x^d=1$, they are
$D_{r,x}(q)=\omega^rQ_{d,x}(\omega^{-r};q)/d$, and we set $E_x=0$. The last term in
\eqref{eq:partial-fractions} is omitted when $\varepsilon_x=0$.
For fixed $x$, all these coefficients are holomorphic for $0<|q|<1$
and have at most a simple pole at $q=0$.
\end{proposition}

\begin{proof}
The rational function in \eqref{eq:partial-fractions} is proper and
its linear denominator factors are distinct. Its coefficient
at $1/u$ is $Q_{d,x}(0;q)=-X_0^{(m)}(x;q)$. To obtain the formulas for $D_{r,x}$ we multiply by 
$1-\omega^ru$, evaluate at $u=\omega^{-r}$ and use
\[
 \lim_{u\to\omega^{-r}}\frac{1-\omega^ru}{1-u^d}=\frac1d
\]
When present, the factor
$1-xu$ similarly gives $E_x$. This also proves uniqueness.
\end{proof}

\section{Elliptic preliminaries}\label{sec:preliminaries}

 In this section, we collect the facts about elliptic functions and Jacobi forms that we require later (for example, see \cite{EichlerZagier,BFOR}).

\subsection{Theta decomposition for holomorphic functions with elliptic law}

Throughout this section, $w\in\CC$ and $u=e^{2\pi i w}$, as in the introduction. For $m\geq 1$ and $\ell\in\ZZ,$ we write
\[
\vartheta_{m,\ell}(\tau,w)
:=
\sum_{\substack{r\in\ZZ\\ r\equiv \ell\, (\mathrm{mod}\,2m)}} q^{r^2/(4m)}u^r.
\]

The following standard lemma records the elliptic half of the Jacobi-form
structure: a holomorphic function satisfying the index \(m\) elliptic
transformation law has a finite expansion in the theta functions
\(\vartheta_{m,\ell}\), with one coefficient for each class
\(\ell \pmod{2m}\) (see \cite[Chapter~5]{EichlerZagier}).

\begin{lemma}[Classical theta decomposition]\label{lem:theta-decomposition}
Let \(m\geq 1\). Suppose that \(F(\tau,w)\) is holomorphic on
\(\HH\times\CC\) and satisfies
\begin{equation}\label{eq:index-m-elliptic}
F(\tau,w+1)=F(\tau,w),
\qquad
F(\tau,w+\tau)=q^{-m}u^{-2m}F(\tau,w).
\end{equation}

Then there exist unique holomorphic coefficient functions \(h_\ell(\tau)\), indexed by \(0\le\ell<2m\), such that
\begin{equation}\label{eq:theta-decomp-formula}
F(\tau,w)=\sum_{\ell=0}^{2m-1} h_\ell(\tau)\,\vartheta_{m,\ell}(\tau,w).
\end{equation}
\end{lemma}
\begin{proof}
Since \(F(\tau,w+1)=F(\tau,w)\), the function \(w\mapsto F(\tau,w)\) has a
Fourier expansion
\[
F(\tau,w)=\sum_{r\in\ZZ} c_r(\tau)u^r,
\]

which converges locally uniformly in \((\tau,w)\). The coefficients may be
written as
\[
c_r(\tau)=\int_0^1 F(\tau,w)e^{-2\pi i r w}\,dw.
\]

Because \(F\) is holomorphic on \(\HH\times\CC\), this integral formula shows
that every \(c_r(\tau)\) is holomorphic in \(\tau\). The local uniform
convergence also justifies the coefficient comparisons below.

We apply the second identity in \eqref{eq:index-m-elliptic}. On the one hand,
\[
F(\tau,w+\tau)=\sum_{r\in\ZZ} c_r(\tau)q^r u^r.
\]

On the other hand, we have
\[
q^{-m}u^{-2m}F(\tau,w)
=
q^{-m}\sum_{r\in\ZZ} c_r(\tau)u^{r-2m}
=
\sum_{r\in\ZZ} q^{-m}c_{r+2m}(\tau)u^r.
\]

Comparing coefficients of \(u^r\) gives
\begin{equation}\label{eq:coeff-recurrence}
c_{r+2m}(\tau)=q^{r+m}c_r(\tau)
\qquad (r\in\ZZ).
\end{equation}
Hence the quantity \(q^{-r^2/4m}c_r(\tau)\) depends only on the residue class
of \(r\) modulo \(2m\), because
\[
q^{-(r+2m)^2/4m}c_{r+2m}(\tau)
=
q^{-r^2/4m-r-m}q^{r+m}c_r(\tau)
=
q^{-r^2/4m}c_r(\tau).
\]
For each residue class \(\ell\pmod{2m}\) define
\[
h_\ell(\tau):=q^{-\ell^2/4m}c_\ell(\tau).
\]
These functions are holomorphic in \(\tau\). Iterating \eqref{eq:coeff-recurrence} shows that for every integer \(n\),
\[
c_{\ell+2mn}(\tau)=h_\ell(\tau)q^{(\ell+2mn)^2/4m}.
\]
Substituting this into the Fourier expansion of \(F\) yields exactly \eqref{eq:theta-decomp-formula}. Uniqueness is immediate from the uniqueness of the Fourier expansion.
\end{proof}

\subsection{Termwise computation of residues}
We next record a lemma that, under suitable convergence hypotheses, justifies computing the residues of an infinite sum of meromorphic functions term by term.

\begin{lemma}
\label{lem:normal-simple-poles}
Let $U\subset\CC^2$ be open with coordinates $(\tau,w)$, and let
$h$ be holomorphic on $U$. Suppose $\partial_wh\ne0$ on
$D=\{h=0\}$. Let $f_n$ be holomorphic on $U\setminus D$, and
suppose each $h f_n$ extends holomorphically to $U$. If
$\sum_n f_n$ converges normally on compact subsets of $U\setminus D$
and $\sum_n h f_n$ converges normally on compact subsets of $U$,
then $\sum_n f_n$ extends meromorphically across $D$ with at most a
simple pole. At $(\tau_0,w_0)\in D$, its residue in $w$ is given by
\[
 \frac{\sum_n(hf_n)(\tau_0,w_0)}{\partial_wh(\tau_0,w_0)}.
\]
\end{lemma}

\begin{proof}
The normally convergent series $N=\sum_n h f_n$ is holomorphic on $U$.
Outside of $D$, it equals $h\sum_n f_n$. Hence, the required continuation is
$N/h$. For fixed $\tau_0$, the denominator has a simple zero at $w_0$,
so its residue is $N(\tau_0,w_0)/\partial_wh(\tau_0,w_0)$.
The numerator is the sum of its values by local uniform convergence.
\end{proof}

\subsection{Appell--Lerch kernels and polar subtraction}

We now define the Appell--Lerch kernels that we require. Such sums and their elliptic and modular properties are studied systematically in \cite{DabholkarMurthyZagier,BFOR}.
Let $m\geq 1$ and let $\beta\in\QQ$. Recall that $u=e^{2\pi i w}$. Define
\begin{equation}\label{eq:appell-constant}
A_{m,\beta}(\tau,w)
:=
-\sum_{n\in\ZZ}
\frac{q^{mn^2}u^{2mn}}{1-e^{-2\pi i\beta}q^n u}.
\end{equation}

The next proposition records the basic analytic properties of these kernels.
They are meromorphic functions with precisely controlled simple poles, and their
elliptic transformation law is exactly the one needed later to cancel the polar
terms of the periodizations $\mathcal E_{m,x}$.

\begin{proposition}\label{prop:appell-constant}
Let $m\geq 1$ and $\beta\in\QQ$.

\begin{enumerate}[label=\textnormal{(\arabic*)}]
\item The series \eqref{eq:appell-constant} converges normally on compact subsets of $\HH\times\CC$ away from its poles and therefore defines a meromorphic function.

\item One has the elliptic transformation law
\begin{equation}\label{eq:appell-elliptic}
A_{m,\beta}(\tau,w+1)=A_{m,\beta}(\tau,w),
\qquad
A_{m,\beta}(\tau,w+\tau)=q^{-m}u^{-2m}A_{m,\beta}(\tau,w).
\end{equation}

\item The poles are simple and lie exactly on the lattice orbit
\[
w\equiv \beta \pmod{\ZZ\tau+\ZZ}.
\]

More precisely,
\begin{equation}\label{eq:appell-local}
A_{m,\beta}(\tau,\beta+\varepsilon)
=
\frac1{2\pi i\,\varepsilon}+O(1)
\qquad (\varepsilon\to 0).
\end{equation}

\end{enumerate}
\end{proposition}

\begin{proof}
Let $K\subset \HH\times\CC$ be a compact set avoiding the poles. Then there is a constant $M>0$ such that $|\operatorname{Im}(w)|\leq M$ and $\operatorname{Im}(\tau)\geq 1/M$ on $K$.

For $n\geq 0$, the denominator in \eqref{eq:appell-constant} is bounded away from $0$ uniformly on $K$, and the numerator satisfies
\[
\bigl|q^{mn^2}u^{2mn}\bigr|
=
\exp\bigl(-2\pi m n^2\operatorname{Im}(\tau)-4\pi mn\operatorname{Im}(w)\bigr)
\leq
\exp\bigl(-2\pi m n^2\operatorname{Im}(\tau)+4\pi mnM\bigr),
\]

which decays like $\exp(-c n^2)$ for some $c>0$.

For $n<0$ write $n=-r$ with $r\geq 1$. Since the compact set avoids the poles, the denominator is again bounded away from $0$ after dividing numerator and denominator by $-e^{-2\pi i\beta}q^{-r}u$:
\[
\frac{q^{mr^2}u^{-2mr}}{1-e^{-2\pi i\beta}q^{-r}u}
=
-\,e^{2\pi i\beta}q^{mr^2+r}u^{-2mr-1}\Bigl(1+O(|q|^r)\Bigr).
\]

The right-hand side also decays like $\exp(-c r^2)$. Thus \eqref{eq:appell-constant} converges normally on compacta away from its poles.

The identity $A_{m,\beta}(\tau,w+1)=A_{m,\beta}(\tau,w)$ is immediate because $u$ is unchanged by $w\mapsto w+1$. For the second transformation law,
\[
A_{m,\beta}(\tau,w+\tau)
=
-\sum_{n\in\ZZ}\frac{q^{mn^2}(qu)^{2mn}}{1-e^{-2\pi i\beta}q^{n+1}u}.
\]

Set $r:=n+1$. Since
\[
mn^2+2mn=m(r-1)^2+2m(r-1)=mr^2-m,
\]
and $u^{2m(r-1)}=u^{2mr-2m}$, we obtain
\[
A_{m,\beta}(\tau,w+\tau)
=
q^{-m}u^{-2m}
\left(-\sum_{r\in\ZZ}\frac{q^{mr^2}u^{2mr}}{1-e^{-2\pi i\beta}q^ru}\right)
=
q^{-m}u^{-2m}A_{m,\beta}(\tau,w).
\]

This proves \eqref{eq:appell-elliptic}.

The denominator in the $n$-th summand vanishes exactly when
$e^{-2\pi i\beta}q^n u=1$,
that is,
\[
w\equiv \beta-n\tau \pmod{\ZZ}.
\]

As $n$ runs through $\ZZ$, these are exactly the points on the orbit $\beta+\ZZ\tau+\ZZ$. Near $w=\beta$, only the term $n=0$ is singular, and it contributes
\[
-\frac{1}{1-e^{2\pi i(w-\beta)}}
=
\frac{1}{2\pi i(w-\beta)}+O(1),
\]

which proves \eqref{eq:appell-local}.
\end{proof}

We now record how to subtract the polar part of such a function to obtain a holomorphic function. This proposition is not needed for the proofs of the main theorems, where the polar subtraction is carried out directly in Proposition~\ref{prop:polar-removal}; we include it to place that construction in the standard framework (see Remark~\ref{rem:bridge}).

\begin{proposition}%[Local polar subtraction for simple poles]
\label{prop:local-polar-subtraction}
Let \(m\geq 1\), let \(S\subset\QQ/\ZZ\) be a finite set of rational
numbers, and for \(\beta\in S\) put
\(\mathcal D_\beta:=\{(\tau,w)\in\HH\times\CC:\ w\in\beta+\ZZ\tau+\ZZ\}\).
Let \(H\) be holomorphic on \((\HH\times\CC)\setminus\bigcup_{\beta\in S}\mathcal D_\beta\)
with at most a simple pole along each \(\mathcal D_\beta\), in the sense that
for all \(a,b\in\ZZ\) the function \((w-\beta-a\tau-b)H(\tau,w)\) extends
holomorphically to a neighbourhood of every point of \(\mathcal D_\beta\)
with \(w=\beta+a\tau+b\). Suppose that \(H\) satisfies the elliptic law
\eqref{eq:index-m-elliptic}.

For each \(\beta\in S\), write the Laurent expansion in the form
\[
H(\tau,\beta+\varepsilon)
=
\frac{\kappa_{\beta}(\tau)}{2\pi i\,\varepsilon}+O(1)
\qquad (\varepsilon\to 0).
\]

Then each residue coefficient \(\kappa_\beta(\tau)\) is holomorphic in \(\tau\).
Define
\[
H^{P,\mathrm{loc}}(\tau,w)
:=
\sum_{\beta\in S}\kappa_{\beta}(\tau)A_{m,\beta}(\tau,w).
\]

Then \(H-H^{P,\mathrm{loc}}\) is holomorphic on \(\HH\times\CC\) and satisfies
the same elliptic law \eqref{eq:index-m-elliptic}. Consequently there are
unique holomorphic coefficient functions \(h_\ell(\tau)\) \((0\le\ell<2m)\)
such that
\begin{equation}\label{eq:local-finite-theta}
H(\tau,w)-H^{P,\mathrm{loc}}(\tau,w)
=
\sum_{\ell=0}^{2m-1} h_\ell(\tau)\vartheta_{m,\ell}(\tau,w).
\end{equation}
\end{proposition}

\begin{proof}
By hypothesis (with \(a=b=0\)), the function
\[
2\pi i\,\varepsilon\,H(\tau,\beta+\varepsilon)
\]
extends to a function holomorphic in \((\tau,\varepsilon)\) near
\(\varepsilon=0\), for every \(\tau\in\HH\). Its value at
\(\varepsilon=0\) is \(\kappa_\beta(\tau)\), so \(\kappa_\beta\) is
holomorphic in \(\tau\). If \(H\) is holomorphic along \(\mathcal D_\beta\),
this value is zero. (Without the joint regularity hypothesis the
conclusion fails: \(H=A_{m,0}(\tau,w)/(\tau-\tau_0)\) satisfies the
elliptic law and has the prescribed poles in \(w\) for \(\tau\ne\tau_0\),
but \(\kappa_0(\tau)=1/(\tau-\tau_0)\).)

Fix \(\beta\in S\). By Proposition~\ref{prop:appell-constant}, the local
expansion of \(\kappa_\beta(\tau)A_{m,\beta}(\tau,\beta+\varepsilon)\) is
\[
\frac{\kappa_\beta(\tau)}{2\pi i\,\varepsilon}+O(1).
\]

Terms associated to other orbit representatives are holomorphic at
\(w=\beta\). Thus \(H\) and \(H^{P,\mathrm{loc}}\) have the same principal
part at $w=\beta$ for every $\beta \in S$, and, since both $H$ and the kernels $A_{m,\beta}$ satisfy the elliptic law \eqref{eq:index-m-elliptic}, the same holds at every translate $\beta+a\tau+b$. Near a point of
\(\mathcal D_\beta\) with \(w=\beta+a\tau+b\), write
\(h:=w-\beta-a\tau-b\); then \(H-H^{P,\mathrm{loc}}=N/h\) with \(N\)
holomorphic in \((\tau,w)\), and the vanishing of the residue means that
\(N=0\) on \(\{h=0\}\). Since \(\partial_wh=1\), we may use \((\tau,h)\)
as local coordinates, and \(N=hN_1\) with \(N_1\) holomorphic; hence
\(H-H^{P,\mathrm{loc}}\) extends holomorphically across
\(\mathcal D_\beta\). Thus \(H-H^{P,\mathrm{loc}}\) is holomorphic on
\(\HH\times\CC\).

Each kernel \(A_{m,\beta}\) satisfies the elliptic law \eqref{eq:index-m-elliptic}, and hence so does \(H^{P,\mathrm{loc}}\). Therefore \(H-H^{P,\mathrm{loc}}\) also satisfies \eqref{eq:index-m-elliptic}. Lemma~\ref{lem:theta-decomposition} gives \eqref{eq:local-finite-theta} with holomorphic coefficient functions.
\end{proof}

\subsection{Constant twists and the Dyson kernel}
For the Dyson periodization we need to accommodate the additional constant multiplier $x^2$.
Write $x=e^{2\pi i\alpha}$ with $\alpha\in\QQ\cap[0,1)$.
If $Y$ satisfies
\[
 Y(\tau,w+1)=Y(\tau,w),\qquad
 Y(\tau,w+\tau)=x^2q^{-m}u^{-2m}Y(\tau,w),
\]
then $\widetilde Y(\tau,w)=Y(\tau,w+\alpha/m)$ satisfies
\eqref{eq:index-m-elliptic}.

The next lemma establishes a finite theta decomposition for holomorphic functions satisfying the twisted index-\(m\) elliptic law in terms of a different basis of theta functions.

\begin{lemma}\label{lem:twisted-theta}
For $0\le\ell<2m$, define
\begin{equation*}%\label{eq:T-def}
 T^{[x]}_{m,\ell}(\tau,w):=
 \sum_{j\in\ZZ}x^{-2j}q^{mj^2+\ell j}u^{\ell+2mj}.
\end{equation*}
These functions are jointly holomorphic and satisfy
\eqref{eq:twisted-law}. Every function $Y$ that is jointly holomorphic on $\HH\times\CC$ and satisfies \eqref{eq:twisted-law} has a unique expansion
\[
 Y(\tau,w)=\sum_{\ell=0}^{2m-1}C_\ell(\tau)T^{[x]}_{m,\ell}(\tau,w),
\]
where the $C_\ell$ are holomorphic on $\HH$.
\end{lemma}

\begin{proof}
The series $T^{[x]}_{m,\ell}(\tau,w)$ converge normally by the standard estimates (as in the proof of Proposition \ref{prop:appell-constant}). Reindexing gives their elliptic transformation behavior.

Since $Y$ is $1$-periodic in $w$,
it descends to a holomorphic function on $\HH\times\CC^\times$.
Write its normally convergent Laurent expansion as
\[
 Y(\tau,w)=\sum_{r\in\ZZ}c_r(\tau)u^r.
\]
By definition each coefficient is a contour integral on a fixed circle in
$\CC^\times$, so it is holomorphic in $\tau$. Comparing Laurent
coefficients in $Y(\tau,w+\tau)=x^2q^{-m}u^{-2m}Y(\tau,w)$ gives
\begin{equation*}%\label{eq:Laurent-recurrence}
 c_{r+2m}=x^{-2}q^{r+m}c_r.
\end{equation*}
For every $j\in\ZZ$, this implies
$c_{\ell+2mj}=x^{-2j}q^{mj^2+\ell j}c_\ell$.
Grouping the Laurent expansion into residue classes modulo $2m$
gives the formula with $C_\ell=c_\ell$.
The coefficient of $u^\ell$ in $T^{[x]}_{m,\ell}$ is $1$, and $T^{[x]}_{m,\ell}$ contains only the powers $u^{\ell+2mj}$ ($j\in\ZZ$), so the supports for different $\ell$ are disjoint, proving uniqueness.
\end{proof}

\begin{remark}\label{rem:bridge}
\begin{enumerate}
 \item

In our normalization, the conversion to the
standard basis is
\begin{equation}\label{eq:T-translation}
 T^{[x]}_{m,\ell}(\tau,w+\alpha/m)
 =e^{2\pi i\ell\alpha/m}q^{-\ell^2/(4m)}
 \vartheta_{m,\ell}(\tau,w).
\end{equation}

\item
The twisted kernel \eqref{eq:K-def} is the standard kernel $A_{m,\beta}$
in translated coordinates. Put $\beta_x=-\alpha-\alpha/m$. Direct
substitution gives
\begin{equation}\label{eq:K-standard-bridge}
 \mathcal K_{m,x}(\tau,w+\alpha/m)=-A_{m,\beta_x}(\tau,w).
\end{equation}
Note that $\mathcal K_{m,x}$ has residue $-1/(2\pi i)$ at
$w=-\alpha$, whereas $A_{m,\beta_x}$ has residue $1/(2\pi i)$ at
$w=\beta_x$. Formula~\eqref{eq:K-standard-bridge} identifies the
polar subtraction in Theorem~\ref{thm:theta} with the general
construction above.

 \end{enumerate}
 \end{remark}

\section{The elliptic completion and its finite part}\label{sec:elliptic}
In this section, we prove Theorem~\ref{thm:elliptic}.
Throughout, we fix $m\ge2$, put $d=2m+1$, and let
$x=e^{2\pi i\alpha}$ with $\alpha\in\QQ\cap[0,1)$.

\subsection{Convergence and meromorphic continuation}

We begin with the estimates that control the infinite sums.
Recall that
\[
 G_d(y;q)=\frac{(y^dq^d;q^d)_\infty}{(yq;q)_\infty}, \qquad g_{k,x}(q)= G_d(xq^k;q).
\]

\begin{lemma}\label{lem:uniform}
For $0<Q<1$, there is a constant $B_Q$ such that
\[
 |g_{k,x}(q)|\le B_Q\qquad(k\ge0,\ |q|\le Q).
\]
Moreover, $g_{k,x}\to1$ locally uniformly as $k\to\infty$.
On every compact subset $S$ of
$(\HH\times\CC)\setminus\mathcal D_x$, there is a constant $A_S$
such that
\begin{equation*}%\label{eq:denominator-bound}
 |1-xuq^j|^{-1}\le A_S\qquad(j\in\ZZ).
\end{equation*}
\end{lemma}

\begin{proof}
The product \eqref{eq:G-product} gives the admissible bound
\[
 B_Q:=\prod_{r\ge1}(1+Q^r+\cdots+Q^{(d-1)r})<\infty.
\]
The tails of this convergent product also prove $g_{k,x}\to1$.
On $S$, choose $|q|\le Q<1$ and $0<a_0\le|u|\le b_0$.
For large positive $j$, we then obtain $|uq^j|\le b_0Q^j\le1/2$.
For large negative $j$, we obtain $|uq^j|\ge a_0Q^j\ge2$.
Both tails give uniformly bounded reciprocal denominators.
Since each of the finitely many remaining denominators is nonzero on $S$, compactness gives their common bound.
\end{proof}

With these bounds in hand, we prove normal convergence of the periodization \(\mathcal E_{m,x}\) and of the kernel \(\mathcal K_{m,x}\). 
Recall that
\[
\mathcal{D}_x= \left\{ (\tau,w)\in \mathbb H \times \mathbb C\, :\, w=-\alpha+a\tau+b, \, a,b\in \mathbb Z\right\}.
\]

\begin{proposition}\label{prop:bilateral-convergence}
The series \(\mathcal E_{m,x}\) defined in \eqref{eq:E-def} and the kernel \(\mathcal K_{m,x}\) defined in \eqref{eq:K-def} converge normally
on compact subsets of $(\HH\times\CC)\setminus\mathcal D_x$.
They extend jointly meromorphically with at most simple poles on
$\mathcal D_x$ and no other poles.
\end{proposition}

\begin{proof}
We substitute \eqref{eq:seed-fractions} into \eqref{eq:E-def} and obtain
\begin{equation}\label{eq:E-double}
 \mathcal E_{m,x}(\tau,w)
 =x^{-2}\sum_{n\in\ZZ}\sum_{k\ge0}
 \frac{x^{-2n}q^{mn^2+k^2}u^{2mn}g_{k,x}(q)}{1-xuq^{n+k}}.
\end{equation}
Let $S$ be a compact set as in Lemma~\ref{lem:uniform} and let
$V=\max(1,b_0^{2m},a_0^{-2m})$ on this set $S$.
The absolute value of each summand is
bounded by
\begin{equation*}%\label{eq:gaussian-majorant}
 A_SB_QQ^{mn^2+k^2}V^{|n|}.
\end{equation*}
Its sum over $n\in\ZZ$ and $k\ge0$ is finite. This proves normal
convergence. Moreover, we can interchange the two sums, and hence identify
\eqref{eq:E-double} with \eqref{eq:E-def}. An analogous argument proves convergence of \eqref{eq:K-def}.

For any point of $\mathcal D_x$, there is exactly one integer $j_0$
for which $1-xuq^{j_0}=0$. 
Thus, we can choose a small neighborhood with compact closure in $\HH\times\CC$  on which no other denominator than 
$1-xuq^{j_0}$ can vanish.
By Lemma~\ref{lem:uniform} only finitely many denominators remain. Each of these, except the one indexed by $j_0$, is
nonzero at the chosen point. By shrinking the neighborhood, we
ensure that all these other denominators stay uniformly away
from zero.

We multiply the series \eqref{eq:E-double} by $1-xuq^{j_0}$.
For every term with $n+k=j_0$, this factor cancels the denominator.
Although there are infinitely many such terms, their numerators
satisfy a summable Gaussian bound. For the remaining terms, both
the reciprocal denominators and the additional factor
$1-xuq^{j_0}$ are uniformly bounded, so the same type of
Gaussian bound applies.

The multiplied series therefore converges normally to a function
holomorphic in $(\tau,w)$. Moreover,
\[
 \partial_w(1-xuq^{j_0})=-2\pi i
 \qquad\text{where }1-xuq^{j_0}=0.
\]
Thus the canceled factor has a simple zero in $w$.
Lemma~\ref{lem:normal-simple-poles} now shows that
$\mathcal E_{m,x}$ extends meromorphically across this zero
with at most a simple pole. The same argument applies to
$\mathcal K_{m,x}$, for which only the summand indexed by
$n=j_0$ can be singular there.
\end{proof}

\subsection{Elliptic law and residues}

We now establish the twisted index-$m$ elliptic laws, compute the residues and complete the proof of Theorem~\ref{thm:elliptic}.

\begin{proposition}\label{prop:law-residues}
Both $\mathcal E_{m,x}$ and $\mathcal K_{m,x}$ satisfy the elliptic transformation law as in
\eqref{eq:twisted-law}. Their residues at
$w_{a,b}=-\alpha+a\tau+b$ are
\begin{align}
 \Res_{w=w_{a,b}}\mathcal E_{m,x}
   &=-\frac{\rho_{m,x}(\tau)}{2\pi i}
            x^{2(m+1)a}q^{-ma^2},\label{eq:res-E-body}\\
 \Res_{w=w_{a,b}}\mathcal K_{m,x}
   &=-\frac{1}{2\pi i}x^{2(m+1)a}q^{-ma^2}.
       \label{eq:res-K-body}
\end{align}
The coefficient $\rho_{m,x}$ is holomorphic on $\HH$.
\end{proposition}

\begin{proof}
Both functions are $1$-periodic in $w$.
Normal convergence permits replacing $n$ by $r-1$ in
\[
 \mathcal E_{m,x}(\tau,w+\tau)
 =\sum_{n\in\ZZ}x^{-2n}q^{mn^2+2mn}u^{2mn}
                      \mathcal G_{d,x}(q^{n+1}u;q).
\]
Since $mn^2+2mn=mr^2-m$, the result is
$x^2q^{-m}u^{-2m}\mathcal E_{m,x}(\tau,w)$.
Exactly the same calculation works for $\mathcal K_{m,x}$.
An iteration gives \eqref{eq:twisted-law}
for all $a,b\in\ZZ$.

The series for $\rho_{m,x}$ converges normally by
Lemma~\ref{lem:uniform}. At $w=-\alpha+\varepsilon$, the singular
terms in \eqref{eq:E-double} are exactly those with $n=-k$.
At $\varepsilon=0$, their numerators are
\[
 x^{-2}x^{2(m+1)k}q^{(m+1)k^2}g_{k,x}(q).
\]
By Proposition~\ref{prop:bilateral-convergence} and Lemma~\ref{lem:normal-simple-poles}
we can sum over these. Since
\[
 \frac1{1-e^{2\pi i\varepsilon}}
       =-\frac1{2\pi i\varepsilon}+O(1),
\]
the residue is $-\rho_{m,x}/(2\pi i)$.
In~\eqref{eq:K-def} only the term $n=0$ is singular at $\varepsilon=0$ and it has residue $-1/(2\pi i)$.
The residues at general $w_{a,b}$ then follow from~\eqref{eq:twisted-law}.
\end{proof}

\begin{proof}[Proof of Theorem~\ref{thm:elliptic}]
Propositions~\ref{prop:bilateral-convergence} and~\ref{prop:law-residues} give the convergence, the meromorphic continuation, the elliptic transformation law, and the residues. It remains to determine the leading term of $\rho_{m,x}$ and to remove the constant twist.

The defining product gives $g_{0,x}=1+O(q)$. 
All terms with $k\ge1$ in \eqref{eq:rho-def} have order at least
$q^{m+1}$, and the uniform estimates show
that their total contribution is $O(q^{m+1})$.
This proves \eqref{eq:rho-leading}.

In particular, $\rho_{m,x}$ is not identically zero, so
$\mathcal E_{m,x}$ has a nonzero residue for some $\tau$
and cannot be identically zero. For any fixed $\tau$,
the residues along the prescribed orbit are nonzero multiples
of $\rho_{m,x}(\tau)$. Thus, if $\rho_{m,x}(\tau)\ne0$,
every point of that orbit is a genuine simple pole.
If $\rho_{m,x}(\tau)=0$, all residues vanish; since the
possible poles are at most simple, they are all removable.

To remove the twist, we translate $w$ by $\alpha/m$.
This replaces $u$ by $e^{2\pi i\alpha/m}u$, so the factor
$u^{-2ma}$ in \eqref{eq:twisted-law} acquires the additional
factor
\[
 e^{-4\pi i\alpha a}=x^{-2a}.
\]
This cancels the multiplier $x^{2a}$ and gives the untwisted
law \eqref{eq:untwisted-law}. The translated function can
have poles only where
\[
 w+\alpha/m=-\alpha+a\tau+b,
\]
so its possible pole orbit is
$-\alpha-\alpha/m+\ZZ\tau+\ZZ$.
\end{proof}

The residue formulas identify exactly which multiple of the kernel $\mathcal{K}_{m,x}$ removes the poles. 
\begin{proposition}\label{prop:polar-removal}
The finite part $\mathcal E^{\mathrm{fin}}_{m,x}$ in
\eqref{eq:finite-def} is jointly holomorphic on $\HH\times\CC$
and satisfies the elliptic transformation law in \eqref{eq:twisted-law}.
\end{proposition}

\begin{proof}
Equations~\eqref{eq:res-E-body} and \eqref{eq:res-K-body} show that $\mathcal E_{m,x}$ and $\rho_{m,x}\mathcal K_{m,x}$ have the
same residue at every possible pole. Their difference therefore
has zero residue there.

To show that the singularities are removable jointly in
$(\tau,w)$, choose the local denominator
$h=1-xuq^{j_0}$ used in the proof of
Proposition~\ref{prop:bilateral-convergence}. Near the chosen
point, we can write
\[
 \mathcal E_{m,x}-\rho_{m,x}\mathcal K_{m,x}
 =\frac{N}{h},
\]
where $N$ is holomorphic in both variables. Since $h$ has a
simple zero in $w$, the residue is $N/\partial_wh$ on $h=0$.
The vanishing of the residue therefore implies that $N=0$
wherever $h=0$.

Because $\partial_wh\ne0$ there, we may use $(\tau,h)$ as
local coordinates. In these coordinates, the Taylor expansion
of $N$ in $h$ has no constant term. Thus $N=hN_1$ for a
holomorphic function $N_1$, and the quotient extends as $N_1$
across $h=0$. This proves joint holomorphy.

Finally, $\rho_{m,x}$ depends only on $\tau$, and both
$\mathcal E_{m,x}$ and $\mathcal K_{m,x}$ satisfy the same
twisted elliptic law. Their difference therefore satisfies
that law as well.
\end{proof}

\section{Theta coefficients of the holomorphic finite part}\label{sec:theta}
In this section, we prove Theorem~\ref{thm:theta}.
We compute the theta coefficients of the holomorphic finite part, prove that they do not vanish, and explain which features depend on the choices made in the construction.

\subsection{Explicit formulas for the theta coefficients}

We first give an explicit formula for the coefficients \(C_{m,\ell,x}\) in the twisted theta expansion of the finite part.

\begin{proposition}\label{prop:coefficients}
The finite part has the expansion
\begin{equation}\label{eq:C-expansion}
 \mathcal E^{\mathrm{fin}}_{m,x}
   =\sum_{\ell=0}^{2m-1}C_{m,\ell,x}(\tau)T^{[x]}_{m,\ell},
\end{equation}
where
\begin{equation}\label{eq:C-explicit}
 C_{m,\ell,x}(\tau)=x^{\ell-2}
   \sum_{k\ge1}g_{k,x}(q)\sum_{s=1}^{k}
        x^{2(m+1)(k-s)}q^{(m+1)k^2-ms^2+\ell s}.
\end{equation}
The series \eqref{eq:C-explicit} converges normally on compact subsets
of $|q|<1$, including $q=0$.
\end{proposition}

\begin{proof}
Proposition~\ref{prop:polar-removal} and
Lemma~\ref{lem:twisted-theta} give a unique expansion in the
twisted theta basis. To compute its coefficients, we separate
\eqref{eq:E-double} according to the index $k$. Define
\[
 \begin{aligned}
 \mathcal E^{(k)}(\tau,w)
 &:=x^{-2}q^{k^2}g_{k,x}(q)
   \sum_{n\in\ZZ}
   \frac{x^{-2n}q^{mn^2}u^{2mn}}{1-xuq^{n+k}},\\
 a_k&:=x^{2(m+1)k-2}q^{(m+1)k^2}g_{k,x}(q).
 \end{aligned}
\]
Thus $\mathcal E_{m,x}=\sum_{k\ge0}\mathcal E^{(k)}$ off
$\mathcal D_x$, and $\rho_{m,x}=\sum_{k\ge0}a_k$.
For fixed $k$, we set $r=n+k$ in the bilateral sum. Factoring
the resulting numerator gives
\begin{equation*}%\label{eq:layer-reindex}
 \mathcal E^{(k)}(\tau,w)
 =a_k\sum_{r\in\ZZ}x^{-2r}q^{mr^2}u^{2mr}
       \frac{(xq^ru)^{-2mk}}{1-xq^ru}.
\end{equation*}
We subtract $a_k\mathcal K_{m,x}$ using the finite geometric identity
\begin{equation*}%\label{eq:finite-geometric}
 \frac{v^{-L}-1}{1-v}=\sum_{j=1}^{L}v^{-j}
 \qquad(L\ge0,\ v\ne0,1).
\end{equation*}
With $L=2mk$ and $v=xq^ru$, this gives
\begin{equation*}%\label{eq:layer-theta}
 \mathcal E^{(k)}-a_k\mathcal K_{m,x}
 =a_k\sum_{j=1}^{2mk}x^{-j}
       \sum_{r\in\ZZ}x^{-2r}q^{mr^2-jr}u^{2mr-j}.
\end{equation*}
Next, we write $j=2ms-\ell$, where $1\le s\le k$ and
$0\le\ell<2m$, and set $r=t+s$ in the inner sum.
By the definition of $T^{[x]}_{m,\ell}$, we obtain
\begin{equation}\label{eq:layer-decomposition}
 \mathcal E^{(k)}-a_k\mathcal K_{m,x}
 =a_k\sum_{\ell=0}^{2m-1}\sum_{s=1}^{k}
   x^{\ell-2(m+1)s}q^{-ms^2+\ell s}T^{[x]}_{m,\ell}.
\end{equation}
For $k=0$, the sum is empty and
$\mathcal E^{(0)}=a_0\mathcal K_{m,x}$. 

To sum over $k$, observe that for $1\le s\le k$ and
$0\le\ell<2m$,
\begin{equation}\label{eq:exponent-lower}
 (m+1)k^2-ms^2+\ell s
 =k^2+m(k^2-s^2)+\ell s\ge k^2.
\end{equation}
Consequently, on $|q|\le Q<1$, the $k$-th contribution to each
coefficient is bounded by $B_QkQ^{k^2}$, using
Lemma~\ref{lem:uniform} and $|x|=1$. 
The coefficient series therefore converge normally on compact
subsets of $|q|<1$. Since their summands have nonnegative integer
powers of $q$ and holomorphic factors $g_{k,x}(q)$, this also
proves holomorphy at $q=0$.

The theta functions are bounded on compact subsets of
$\HH\times\CC$, so the right-hand side of
\eqref{eq:layer-decomposition} can also be summed normally there.
On the left, we established normal convergence off $\mathcal D_x$ 
in Proposition~\ref{prop:bilateral-convergence}, and the series for
$\rho_{m,x}$ converges normally. Summing over $k$ and substituting
the definition of $a_k$ now gives
\eqref{eq:C-expansion}--\eqref{eq:C-explicit} off $\mathcal D_x$.
Both sides are jointly holomorphic, so the identity extends across
$\mathcal D_x$.
\end{proof}

\subsection{Leading terms and nonvanishing}
The contributions with $k=0$ and $k=1$ make the theta decomposition concrete.
For $k=0$, there is no theta contribution, since
\[
 \mathcal E^{(0)}=x^{-2}g_{0,x}\mathcal K_{m,x}.
\]
For $k=1$, the layer formula gives
\begin{equation*}%\label{eq:first-theta-layer}
 \mathcal E^{(1)}-x^{2m}q^{m+1}g_{1,x}\mathcal K_{m,x}
   =g_{1,x}\sum_{\ell=0}^{2m-1}
                  x^{\ell-2}q^{\ell+1}T^{[x]}_{m,\ell}.
\end{equation*}
Thus every theta component is already present in the first nonpolar
layer. The next proposition shows that the terms with
$k\ge2$ do not affect these leading terms.

\begin{proposition}\label{prop:nontrivial}
For $0\le\ell<2m$, we have
\begin{equation}\label{eq:C-leading}
 C_{m,\ell,x}(\tau)=x^{\ell-2}q^{\ell+1}(1+O(q^2)).
\end{equation}
For $x=1$, every $C_{m,\ell,1}$ has nonnegative integer power-series
coefficients and is strictly positive for real $0<q<1$.
\end{proposition}

\begin{proof}
In \eqref{eq:C-explicit}, the term $k=s=1$ is
$x^{\ell-2}q^{\ell+1}g_{1,x}(q)$.
The product \eqref{eq:G-product} gives $g_{1,x}(q)=1+O(q^2)$.
All other terms have $k\ge2$, and their exponents satisfy
\[
 (m+1)k^2-ms^2+\ell s
 =k^2+\ell+m(k^2-s^2)+\ell(s-1)\ge k^2+\ell.
\]
The uniform bounds in Proposition~\ref{prop:coefficients} show that
their sum is $O(q^{\ell+4})$. This proves \eqref{eq:C-leading}.
When $x=1$, each $g_{k,1}$ is a product of polynomials with
nonnegative integer coefficients and constant term $1$.
Equation~\eqref{eq:exponent-lower} implies that only finitely many
$k$ and factors contribute to any prescribed power of $q$.
Thus \eqref{eq:C-explicit} has nonnegative integer coefficients.
Its leading coefficient is $1$, giving strict positivity on $(0,1)$.
\end{proof}

The nonnegative coefficients also have a combinatorial interpretation.
Let $\mathcal P_{d,k}$ denote the set of partitions whose parts exceed $k$, and in which each part occurs at most $d-1$ times. By the factorization of \eqref{eq:G-product}, $g_{k,1}$ is the generating function of $\mathcal P_{d,k}$. 
The
coefficient of $q^N$ in $C_{m,\ell,1}$ counts triples $(k,s,\lambda)$
with
\[
 k\ge1,\qquad 1\le s\le k,\qquad \lambda\in\mathcal P_{d,k},
 \qquad
 N=(m+1)k^2-ms^2+\ell s+|\lambda|.
\]
The bound \eqref{eq:exponent-lower} makes this a finite count for each
$N$. We emphasize that these are auxiliary counts attached to the theta coefficients, not the
full-rank counts $N_m(r,N)$ of the Dyson symbols themselves. The latter
are encoded by \eqref{eq:intro-Rm}, whereas the former only arise after
periodization and  polar subtraction.
With these results, we now prove Theorem~\ref{thm:theta} by translating to the standard theta basis.

\begin{proof}[Proof of Theorem~\ref{thm:theta}]
Joint holomorphy and the twisted elliptic law follow from
Proposition~\ref{prop:polar-removal}. Translation removes the twist
as in Theorem~\ref{thm:elliptic}. Applying \eqref{eq:T-translation} to
\eqref{eq:C-expansion} gives \eqref{eq:theta-expansion} with
\[
 h_{m,\ell,x}=e^{2\pi i\ell\alpha/m}q^{-\ell^2/(4m)}C_{m,\ell,x}.
\]
This proves the explicit formula, convergence, and holomorphy.
Uniqueness follows from Lemma~\ref{lem:twisted-theta}, and
\eqref{eq:C-leading} gives \eqref{eq:h-leading}.
In particular, none of the coefficient functions is identically zero.
Linear independence of the theta basis implies the same for the
finite part. For $x=1$ and $\tau=it$, the factor
$q^{-\ell^2/(4m)}$ is positive, so positivity follows from
Proposition~\ref{prop:nontrivial}.
\end{proof}

\subsection{Choice of index and polar kernel}
We conclude this section with remarks on the choices of index and polar kernel made in the construction.

The index of the periodization is a choice. Namely, if the
seed $\mathcal G_{d,x}$ is held fixed and $M\ge1$ is any integer, then
\[
 \sum_{n\in\ZZ}x^{-2n}q^{Mn^2}u^{2Mn}\mathcal G_{d,x}(q^nu;q)
\]
converges by the same argument and satisfies the twisted index-$M$
elliptic law. In \eqref{eq:E-def} we choose $M=m$.
We note that the finite recurrence describes the seed, but it does not single out this choice of index.

The kernel \eqref{eq:K-def} fixes the normalization of the polar part.
If a holomorphic theta combination $S$ satisfying the same twisted law is added
to $\mathcal K_{m,x}$, then the finite part becomes
$\mathcal E^{\mathrm{fin}}_{m,x}-\rho_{m,x}S$.
The connection to the higher Dyson series is the exact seed identity \eqref{eq:seed-specialization}, and the periodization is a further operation on that seed. We note that the seed occurs as the $n=0$ summand of the periodization, but this does not identify their Laurent expansions or their residues.

\section{The recurrence as a characterization of the Dyson seed}
\label{sec:remark-proof}

In this section, we prove the three claims of
Theorem~\ref{rem:recurrence-implies}. The idea is simple. Namely, the finite
recurrence \eqref{eq:recurrence} admits an explicit basis of
Lambert-type solutions, one for each nontrivial $d$-th root of unity.
The Dyson sequence $X_j^{(m)}(x;q)/G_d(x;q)$ is singled out among all
solutions by three properties: its limit $1$, the location of the poles of its
periodized seed, and, at $x=1$, the $q$-adic approximation of
$\RRm_m(\zeta_d;q)$.

Throughout this section, we fix $m\ge2$, $d=2m+1$, a root of unity
$x=e^{2\pi i\alpha}$ with $\alpha\in\QQ\cap[0,1)$, and $\tau\in\HH$,
and we put $q=e^{2\pi i\tau}$, so that $0<|q|<1$. As in Proposition~\ref{prop:partial-fractions}, we write $\omega:=e^{2\pi i/d}$, and we use the notation $C_y=\sum_{r=0}^{d-1}y^r$ of Theorem~\ref{thm:finite} for every $y\in\CC$, so that
\begin{equation}\label{eq:P-def}
 (1-y)C_y=1-y^d.
\end{equation}
Thus $C_y=0$ if and only if $y\in\{\omega,\omega^2,\dots,\omega^{d-1}\}$,
and $C_x=0$ if and only if $x=\omega^{k_0}$ for some $1\le k_0\le d-1$.
For $c\in\CC$ with $|c|=1$ and $i\ge0$, we put
\begin{equation}\label{eq:hratio-def}
 h_i(c):=\frac{(cq;q)_i}{(c^dq^d;q^d)_i},
 \qquad\text{and, for }0<Q<1,\qquad
 B'_Q:=\prod_{l\ge1}\frac{1+Q^l}{1-Q^{dl}}<\infty.
\end{equation}
Since $|1-cq^l|\le1+|q|^l$ and $|1-c^dq^{dl}|\ge1-|q|^{dl}$, we have
\begin{equation}\label{eq:hratio-bound}
 |h_i(c)|\le B'_Q\qquad(|c|=1,\ i\ge0,\ |q|\le Q).
\end{equation}
We note that $h_i(x)=G_d(xq^i;q)/G_d(x;q)=g_{i,x}(q)/g_{0,x}(q)$ by
\eqref{eq:G-product}. Here $G_d(x;q)\ne0$ because $|x^dq^{dl}|<1$ for
$l\ge1$. For $\gamma\in\CC$, let $V_\gamma$ denote the set of sequences
$Y=(Y_j)_{j\ge0}$ of complex numbers satisfying
\begin{equation}\label{eq:Y-recurrence}
 \sum_{r=0}^{d-1}x^rY_{j+r-2}-q^{j-1}Y_j=C_x\gamma\qquad(j\ge2).
\end{equation}
This is \eqref{eq:recurrence} with $X_j^{(m)}(x;q)$ replaced by $Y_j$
and $G_d(x;q)$ by $\gamma$. We let $V:=\bigcup_{\gamma\in\CC}V_\gamma$ be
the family considered in Theorem~\ref{rem:recurrence-implies}. Since
\eqref{eq:Y-recurrence} is linear in the pair $(Y,\gamma)$, we have
$aY+bY'\in V_{a\gamma+b\gamma'}$ whenever $Y\in V_\gamma$,
$Y'\in V_{\gamma'}$, and $a,b\in\CC$. For $Y\in V$, we define the seed
and its bilateral periodization, exactly as in \eqref{eq:seed-def} and
\eqref{eq:E-def}, by
\begin{equation}\label{eq:EY-def}
 \mathcal G_Y(u):=x^{-2}\sum_{j\ge0}Y_j(xu)^j,\qquad
 \mathcal E_Y(\tau,w):=\sum_{n\in\ZZ}x^{-2n}q^{mn^2}u^{2mn}
 \mathcal G_Y(q^nu).
\end{equation}
Finally, generalizing $\mathcal D_x$ of Section~1, for $c=e^{2\pi i\beta}$ with $\beta\in\QQ\cap[0,1)$, we write
\begin{equation}\label{eq:Dc-def}
 \mathcal D_c:=\{(\tau,w)\in\HH\times\CC:\ w=-\beta+a\tau+b,\ a,b\in\ZZ\},
\end{equation}
so that $\mathcal D_{\omega^k}$ is the orbit of $w=-k/d$. For fixed $\tau$, we
also call the fibre $\{w:(\tau,w)\in\mathcal D_c\}$ an orbit. Two such
sets $\mathcal D_c$, $\mathcal D_{c'}$ coincide if $c=c'$ and are disjoint
otherwise. In particular, $\mathcal D_x=\mathcal D_{\omega^k}$ if and only
if $x=\omega^k$.

\begin{theorem}\label{thm:remark-proof}
Under the standing assumptions of this section, and with the
notation \eqref{eq:Y-recurrence}--\eqref{eq:Dc-def}, the following hold.

\begin{enumerate}[label=\textup{(\roman*)}]
\item The sequence $Y_j=X_j^{(m)}(x;q)/G_d(x;q)$ belongs to $V_1$ and
tends to $1$. Conversely, if $Y\in V_\gamma$ and $Y_j\to1$, then
$Y_j=X_j^{(m)}(x;q)/G_d(x;q)$ for all $j\ge0$, and $\gamma=1$ unless
$C_x=0$. In particular, $X^{(m)}(x;q)/G_d(x;q)$ is the unique sequence
in $V$ tending to $1$. This includes the case $C_x=0$, in which
$V_\gamma=V_0$ for every $\gamma$.
\item Let $Y\in V$. Then the following hold.
\begin{enumerate}[label=\textup{(\alph*)}]
\item \emph{The seed.} The seed $\mathcal G_Y(u)$ converges for $|u|<1$ and
extends to a meromorphic function on $\CC$ with at most simple poles,
all contained in the set
$\{x^{-1}q^{-i},\ \omega^{-k}q^{-i}:\ i\ge0,\ 1\le k\le d-1\}$.
\item \emph{The periodization for fixed $\tau$.} The bilateral periodization $w\mapsto\mathcal E_Y(\tau,w)$
converges normally on compact subsets of the complement in $\CC$ of the
orbits $-\alpha+\ZZ\tau+\ZZ$ and $-k/d+\ZZ\tau+\ZZ$ ($1\le k\le d-1$),
extends to a meromorphic function of $w$ with at most simple poles on
these orbits, and satisfies the twisted law \eqref{eq:twisted-law}.
\item \emph{Holomorphic families.} Suppose that $Y$ is given as a function of $q$ such that the
initial values $Y_0,\dots,Y_{d-2}$ and the parameter $\gamma$ are
holomorphic on an open set $\Omega\subset\{q:0<|q|<1\}$, and put $\widetilde\Omega:=\{\tau\in\HH:e^{2\pi i\tau}\in\Omega\}$. Then
$\mathcal E_Y$ converges normally on compact subsets of
$(\widetilde\Omega\times\CC)\setminus
(\mathcal D_x\cup\bigcup_{k=1}^{d-1}\mathcal D_{\omega^k})$ and extends
jointly meromorphically to $\widetilde\Omega\times\CC$ with at most
simple poles on these orbits.
\item \emph{The location of the poles.} There is a discrete subset $\Xi_m$ of $\{q:0<|q|<1\}$,
depending only on $m$ and disjoint from $\{q:0<|q|\le1/2\}$, with the
following property. If $q\notin\Xi_m$, then the poles of
$w\mapsto\mathcal E_Y(\tau,w)$ are contained in $\mathcal D_x$ if and
only if $Y=c\,X^{(m)}(x;q)$ for some $c\in\CC$. Moreover, for $q\notin\Xi_m$ and
every $k$ with $\omega^k\ne x$, there is a sequence in $V_0$ that is not a
scalar multiple of $X^{(m)}(x;q)$ and whose periodization has
simple poles exactly on the single orbit $\mathcal D_{\omega^k}$.
\end{enumerate}
\item Let $x=1$ and $\gamma=1$, and let $J\ge0$ be an integer. We put
$Y_J=Y_{J+1}=\dots=Y_{J+2m-1}=1$, and we define $Y_{J-1},Y_{J-2},\dots,Y_0$
successively by solving \eqref{eq:Y-recurrence} backwards, that is,
\begin{equation}\label{eq:backward}
 Y_{j-2}:=d+q^{j-1}Y_j-\sum_{r=1}^{d-1}Y_{j+r-2}
 \qquad(j=J+1,J,\dots,2).
\end{equation}
Then $P_J(q):=Y_0$ is a polynomial in $\ZZ[q]$, and
\[
 \RRm_m(\zeta_d;q)\equiv P_J(q)\pmod{q^{J+1}}.
\]
\end{enumerate}
\end{theorem}

We comment on the terminal data in (iii). Since the coefficient of $Y_{j-2}$ in \eqref{eq:Y-recurrence} is
$1$, each backward step \eqref{eq:backward} consumes the $d-1=2m$
consecutive values $Y_{j-1},\dots,Y_{j+d-3}$. Therefore the $2m$ prescribed
terminal values in (iii) are exactly the data needed to start
the backward solution. The sequence so obtained is the restriction
to $0\le j\le J+d-2$ of the unique $Y\in V_1$ with
$Y_J=\dots=Y_{J+d-2}=1$ (see Lemma~\ref{lem:solution-space}). For
instance, for $m=2$ and $J=4$, one finds
$P_4(q)=1+q-q^2+q^4-q^5-q^6$, which agrees with
$\RRm_2(\zeta_5;q)=1+q-q^2+q^4+O(q^5)$.

The proof of Theorem~\ref{thm:remark-proof} occupies the rest of this section. We begin with the
explicit solutions. The following lemma constructs, for each $c$ of modulus $1$, a Lambert-type sequence $S_j(c)$ that satisfies the recurrence with right-hand side $\lambda^{j-2}C_c$.

\begin{lemma}\label{lem:lambert-solutions}
Let $c\in\CC$ with $|c|=1$, and put $\lambda:=c/x$. We define
\begin{equation}\label{eq:S-def}
 S_j(c):=\lambda^j\sum_{i\ge0}\lambda^{2i}q^{i^2+ij}h_i(c)
 \qquad(j\in\ZZ).
\end{equation}
\begin{enumerate}[label=\textup{(\alph*)}]
\item The series converges absolutely for every $j\in\ZZ$, and for
$0<Q<1$, $|q|\le Q$, and $j\ge0$ we have
\begin{equation}\label{eq:S-asymp}
 \bigl|\lambda^{-j}S_j(c)-1\bigr|\le\frac{B'_Q}{1-Q}\,|q|^{j+1}.
\end{equation}
\item For every $j\in\ZZ$, we have
\begin{equation}\label{eq:S-recurrence}
 \sum_{r=0}^{d-1}x^rS_{j+r-2}(c)-q^{j-1}S_j(c)=\lambda^{j-2}C_c.
\end{equation}

\item We have $S_j(x)=X_j^{(m)}(x;q)/G_d(x;q)$ for all $j\in\ZZ$.
\item We put $b_i(c):=\lambda^{2i}q^{i^2}h_i(c)$. For $|u|<1$, we have
\begin{equation}\label{eq:S-seed}
 x^{-2}\sum_{j\ge0}S_j(c)(xu)^j
 =x^{-2}\sum_{i\ge0}\frac{b_i(c)}{1-cq^iu},
\end{equation}
and the right-hand side is a meromorphic function of $u\in\CC$ with at
most simple poles, at the points $u=c^{-1}q^{-i}$, $i\ge0$.
\end{enumerate}
\end{lemma}

\begin{proof}
(a) By \eqref{eq:hratio-bound} and $|\lambda|=1$, the $i$-th term of
\eqref{eq:S-def} is bounded by $B'_Q|q|^{i^2+ij}$, and
$i^2+ij\to\infty$ as $i\to\infty$ for every fixed $j\in\ZZ$. This
gives absolute convergence. The $i=0$ term is $\lambda^j$. For $j\ge0$
and $i\ge1$, we have $i^2+ij-(j+1)=(i-1)(i+1+j)\ge i-1$. Hence
$\sum_{i\ge1}|q|^{i^2+ij}\le|q|^{j+1}\sum_{i\ge1}Q^{i-1}$. This
proves \eqref{eq:S-asymp}.

(b) All series below converge absolutely, so they may be
rearranged freely. Using $x\lambda=c$, we obtain
\[
 \sum_{r=0}^{d-1}x^rS_{j+r-2}(c)
 =\sum_{i\ge0}\lambda^{j-2}\lambda^{2i}q^{i^2+i(j-2)}h_i(c)
  \sum_{r=0}^{d-1}(x\lambda q^i)^r
 =\sum_{i\ge0}\lambda^{j-2}\lambda^{2i}q^{i^2+i(j-2)}h_i(c)\,C_{cq^i}.
\]

On the other hand, shifting the summation index $i\mapsto i-1$ and
using $(i-1)^2+(i-1)j+j-1=i^2+i(j-2)$, we obtain
\[
 q^{j-1}S_j(c)
 =\sum_{i\ge0}\lambda^{j}\lambda^{2i}q^{i^2+ij+j-1}h_i(c)
 =\sum_{i\ge1}\lambda^{j-2}\lambda^{2i}q^{i^2+i(j-2)}h_{i-1}(c).
\]
For $i\ge1$, the identity $(1-y)C_y=1-y^d$ with $y=cq^i$ and the
definition \eqref{eq:hratio-def} give
\[
 h_i(c)C_{cq^i}
 =h_{i-1}(c)\,\frac{1-cq^i}{1-c^dq^{di}}\cdot\frac{1-c^dq^{di}}{1-cq^i}
 =h_{i-1}(c).
\]
Here we use that $cq^i\ne1$ and $c^dq^{di}\ne1$, because $|q|<1$. Hence all
terms with $i\ge1$ cancel in the difference, and the $i=0$ term
contributes $\lambda^{j-2}h_0(c)C_c=\lambda^{j-2}C_c$. This proves
\eqref{eq:S-recurrence}.

(c) For $c=x$, we have $\lambda=1$ and
$h_i(x)=G_d(xq^i;q)/G_d(x;q)$, so \eqref{eq:S-def} becomes
$S_j(x)=\sum_{i\ge0}q^{i^2+ij}G_d(xq^i;q)/G_d(x;q)=X_j^{(m)}(x;q)/G_d(x;q)$. This proves (c).

(d) For $|u|<1$, the double series
$\sum_{i,j\ge0}b_i(c)\lambda^j(xu)^jq^{ij}$ converges absolutely, since
its terms are bounded by $|b_i(c)|\,|u|^j$. Summing first over $j$ and
using $\lambda xq^iu=cq^iu$ gives \eqref{eq:S-seed}. On a compact subset
of $\CC$ avoiding the points $c^{-1}q^{-i}$, the quantities
$|1-cq^iu|^{-1}$ are bounded uniformly in $i$, and
$|b_i(c)|\le B'_Q|q|^{i^2}$. Hence the right-hand side of
\eqref{eq:S-seed} converges normally there and is meromorphic on $\CC$
with at most simple poles at the points $c^{-1}q^{-i}$, exactly as in
Lemma~\ref{lem:seed}. This proves (d).
\end{proof}

For $1\le k\le d-1$, we put
\begin{equation}\label{eq:Z-def}
 \lambda_k:=\omega^k/x,\qquad
 Z^{(k)}:=\bigl(S_j(\omega^k)\bigr)_{j\ge0}.
\end{equation}

The numbers $\lambda_1,\dots,\lambda_{d-1}$ are distinct and of modulus
$1$, and $\lambda_k=1$ if and only if $x=\omega^k$. We now show that the
sequences $Z^{(k)}$ form a basis of the homogeneous solution space $V_0$.

\begin{lemma}\label{lem:solution-space}
\begin{enumerate}[label=\textup{(\alph*)}]
\item For every $\gamma\in\CC$ and every $J\ge0$, the map
$V_\gamma\to\CC^{d-1}$, $Y\mapsto(Y_J,\dots,Y_{J+d-2})$, is a bijection.
In particular, $V_0$ is a complex vector space of dimension $d-1$, and
$V_\gamma=Y^0+V_0$ for any $Y^0\in V_\gamma$.
\item The sequences $Z^{(1)},\dots,Z^{(d-1)}$ form a basis of $V_0$.
\item We have $X^{(m)}(x;q)/G_d(x;q)\in V_1$. If $C_x\ne0$, then every
$Y\in V_\gamma$ can be written uniquely as
\begin{equation}\label{eq:Y-decomposition}
 Y=\gamma\,\frac{X^{(m)}(x;q)}{G_d(x;q)}+\sum_{k=1}^{d-1}s_kZ^{(k)}
 \qquad(s_k\in\CC).
\end{equation}
If $C_x=0$, then $x=\omega^{k_0}$ for some $1\le k_0\le d-1$, we have
$V_\gamma=V_0$ for every $\gamma$ and
$Z^{(k_0)}=X^{(m)}(x;q)/G_d(x;q)$, and every $Y\in V_\gamma$ is uniquely
of the form $Y=\sum_{k=1}^{d-1}s_kZ^{(k)}$.
\end{enumerate}
\end{lemma}

\begin{proof}
(a) For $j\ge2$, equation \eqref{eq:Y-recurrence} can be solved
for its term of highest index, whose coefficient is $x^{d-1}\ne0$, and
for its term of lowest index, whose coefficient is $1$:
\[
 x^{d-1}Y_{j+d-3}=C_x\gamma+q^{j-1}Y_j-\sum_{r=0}^{d-2}x^rY_{j+r-2},
 \qquad
 Y_{j-2}=C_x\gamma+q^{j-1}Y_j-\sum_{r=1}^{d-1}x^rY_{j+r-2}.
\]
The right-hand sides involve only the $d-1$ indices $j-2,\dots,j+d-4$ and
 $j-1,\dots,j+d-3$, respectively (both ranges include $j$, since $d\ge5$). Hence,
given $Y_J,\dots,Y_{J+d-2}$, the first formula with $j=J+2,J+3,\dots$
determines $Y_{J+d-1},Y_{J+d},\dots$ successively, and the second one
with $j=J+1,J,\dots,2$ determines $Y_{J-1},\dots,Y_0$. Conversely, the
sequence so defined satisfies \eqref{eq:Y-recurrence} for all $j\ge2$.
This proves the bijectivity. The remaining assertions follow from
the linearity of \eqref{eq:Y-recurrence} in $(Y,\gamma)$.

(b) By \eqref{eq:S-recurrence} with $c=\omega^k$ and
$C_{\omega^k}=0$, each $Z^{(k)}$ lies in $V_0$. Suppose that
$\sum_{k=1}^{d-1}s_kZ^{(k)}_j=0$ for all $j\ge0$. By \eqref{eq:S-asymp}, we have
$Z^{(k)}_j=\lambda_k^j+O(|q|^{j+1})$, and so
$\sum_ks_k\lambda_k^j\to0$ as $j\to\infty$. Since the $\lambda_k$ are
distinct and of modulus $1$, the uniqueness theorem for the Fourier coefficients of almost periodic sequences (see \cite[Chapter~I]{Besicovitch} or \cite[Chapter~1]{Corduneanu}) forces
$s_1=\dots=s_{d-1}=0$. In the present elementary situation this amounts to the observation that, for each $k'$, the Ces\`aro means $\frac1N\sum_{j=0}^{N-1}\overline{\lambda_{k'}}^{\,j}\sum_ks_k\lambda_k^{\,j}$ tend to $s_{k'}$ as $N\to\infty$, because $\frac1N\sum_{j=0}^{N-1}\mu^j\to0$ for $|\mu|=1$, $\mu\ne1$, whereas they tend to $0$ by hypothesis. Thus the $d-1$ sequences $Z^{(k)}$ are linearly
independent in the $(d-1)$-dimensional space $V_0$, and hence they form a basis. This proves (b).

(c) Lemma~\ref{lem:lambert-solutions}(b),(c) with $c=x$ and
$\lambda=1$ show that $X^{(m)}(x;q)/G_d(x;q)$ lies in
$V_1$. If $C_x\ne0$, then linearity gives
$\gamma X^{(m)}(x;q)/G_d(x;q)\in V_\gamma$, so (a) and (b) yield
\eqref{eq:Y-decomposition}. If $C_x=0$, then $x^d=1$ and $x\ne1$, i.e.\
$x=\omega^{k_0}$ with $1\le k_0\le d-1$. In this case, the right-hand side of
\eqref{eq:Y-recurrence} vanishes for every $\gamma$, so $V_\gamma=V_0$.
Moreover, $S_j(x)=S_j(\omega^{k_0})$ gives $Z^{(k_0)}=X^{(m)}(x;q)/G_d(x;q)$
by Lemma~\ref{lem:lambert-solutions}(c). The remaining assertion follows from (b).
\end{proof}

Next, we record that the analytic results of
Section~\ref{sec:elliptic} apply verbatim to the periodization of any
Lambert-type seed with Gaussian coefficient decay. The Dyson seed
\eqref{eq:seed-fractions} is the case $c=x$ and $b_i=q^{i^2}g_{i,x}(q)$.

\begin{lemma}\label{lem:lambert-periodization}
Let $c=e^{2\pi i\beta}$ with $\beta\in\QQ\cap[0,1)$, and let
$(b_i)_{i\ge0}$ be holomorphic functions of $q$ on $|q|<1$ such that, for
every $0<Q<1$, there is a constant $B_Q$ with $|b_i(q)|\le B_Q|q|^{i^2}$
for all $i\ge0$ and $|q|\le Q$. We define
\begin{equation}\label{eq:EPhi-def}
 \Phi(u):=x^{-2}\sum_{i\ge0}\frac{b_i}{1-cq^iu},\qquad
 \mathcal E_\Phi(\tau,w):=\sum_{n\in\ZZ}x^{-2n}q^{mn^2}u^{2mn}\Phi(q^nu),
\end{equation}
and
\begin{equation}\label{eq:rhoPhi-def}
 \rho_\Phi(\tau):=x^{-2}\sum_{i\ge0}x^{2i}c^{2mi}q^{mi^2}b_i.
\end{equation}
\begin{enumerate}[label=\textup{(\alph*)}]
\item The series $\mathcal E_\Phi$ converges normally on compact subsets of
$(\HH\times\CC)\setminus\mathcal D_c$, extends jointly meromorphically
to $\HH\times\CC$ with at most simple poles on $\mathcal D_c$ and no
other poles, and satisfies the twisted law \eqref{eq:twisted-law}.
\item The series $\rho_\Phi$ converges normally on compact subsets of
$|q|<1$, and so $\rho_\Phi$ is holomorphic on $\HH$. For all
$a,b\in\ZZ$, we have
\begin{equation}\label{eq:res-Phi}
 \Res_{w=-\beta+a\tau+b}\mathcal E_\Phi(\tau,w)
 =-\frac{\rho_\Phi(\tau)}{2\pi i}\,x^{2a}c^{2ma}q^{-ma^2}.
\end{equation}
\item Fix $\tau\in\HH$. If $\rho_\Phi(\tau)\ne0$, then every point of
the orbit $-\beta+\ZZ\tau+\ZZ$ is a simple pole of
$\mathcal E_\Phi(\tau,\cdot)$. If $\rho_\Phi(\tau)=0$, then all these
singularities are removable.
\end{enumerate}
\end{lemma}

\begin{proof}
The proofs of Lemma~\ref{lem:uniform},
Propositions~\ref{prop:bilateral-convergence} and
\ref{prop:law-residues}, and Theorem~\ref{thm:elliptic}(iii) use only
$|x|=1$, the fact that the pole parameter $x$ in \eqref{eq:seed-fractions}
has modulus $1$, and the bound $|q^{k^2}g_{k,x}(q)|\le B_Q|q|^{k^2}$ on
its coefficients. We indicate the necessary changes. Substituting
\eqref{eq:EPhi-def} gives, as in \eqref{eq:E-double},
\[
 \mathcal E_\Phi(\tau,w)
 =x^{-2}\sum_{n\in\ZZ}\sum_{i\ge0}
 \frac{x^{-2n}q^{mn^2}u^{2mn}\,b_i}{1-cuq^{n+i}}.
\]
On a compact subset $S$ of $(\HH\times\CC)\setminus\mathcal D_c$ with
$|q|\le Q$ and $0<a_0\le|u|\le b_0$, the argument of
Lemma~\ref{lem:uniform} with $x$ replaced by $c$ gives
$|1-cuq^j|^{-1}\le A_S$ for all $j\in\ZZ$. Hence each summand is bounded by
$A_SB_QQ^{mn^2+i^2}\max(1,b_0^{2m},a_0^{-2m})^{|n|}$. This proves
normal convergence and justifies the interchange of the two sums.
Exactly one factor $1-cuq^{j_0}$ vanishes at a given point of
$\mathcal D_c$. Multiplying by this factor and applying
Lemma~\ref{lem:normal-simple-poles}, as in the proof of
Proposition~\ref{prop:bilateral-convergence}, gives the meromorphic
continuation with at most simple poles on $\mathcal D_c$. The twisted
law follows by the reindexing $n=r-1$ in the proof of
Proposition~\ref{prop:law-residues}, which involves only the weight
$x^{-2n}q^{mn^2}u^{2mn}$ and not the seed. This proves (a).

We now prove (b). The bound
$|x^{2i}c^{2mi}q^{mi^2}b_i|\le B_QQ^{(m+1)i^2}$ on $|q|\le Q$ gives normal
convergence of $\rho_\Phi$ on compact subsets of $|q|<1$. At
$w=-\beta+\varepsilon$, we have $u=c^{-1}e^{2\pi i\varepsilon}$, and so the
denominators $1-cuq^{n+i}=1-q^{n+i}e^{2\pi i\varepsilon}$ vanish at
$\varepsilon=0$ exactly for $n=-i$. The numerators of these terms at
$\varepsilon=0$ are $x^{-2}x^{2i}q^{mi^2}c^{2mi}b_i$, and
$1/(1-e^{2\pi i\varepsilon})=-1/(2\pi i\varepsilon)+O(1)$. Summing over $i$,
which Lemma~\ref{lem:normal-simple-poles} permits, gives
$\Res_{w=-\beta}\mathcal E_\Phi=-\rho_\Phi/(2\pi i)$. For general $a,b$, we
apply \eqref{eq:twisted-law}. The residue in $w$ is invariant under
$w\mapsto w+a\tau+b$, and the factor $u^{-2ma}$ equals $c^{2ma}$ at
$w=-\beta$. This proves \eqref{eq:res-Phi}. Finally, part (c) follows from (b) as
in the proof of Theorem~\ref{thm:elliptic}(iii), because the residues
along the orbit are nonzero multiples of $\rho_\Phi(\tau)$ and the poles
are at most simple. This proves the lemma.
\end{proof}

We now apply Lemma~\ref{lem:lambert-periodization} to the seeds of the eigen-solutions $Z^{(k)}$. By
Lemma~\ref{lem:lambert-solutions}(d) with $c=\omega^k$, the seed of
$Z^{(k)}$ is
\begin{equation}\label{eq:Phi-k}
 \mathcal G_{Z^{(k)}}(u)
 =x^{-2}\sum_{i\ge0}\frac{b_i(\omega^k)}{1-\omega^kq^iu}.
\end{equation}
Each $b_i(\omega^k)$ is a rational function of $q$ whose denominator
$(q^d;q^d)_i$ (by \eqref{eq:hratio-def}, as $\omega^{kd}=1$) does not vanish on $|q|<1$. Hence $b_i(\omega^k)$ is holomorphic there,
and $|b_i(\omega^k)|\le B'_Q|q|^{i^2}$ for $|q|\le Q$ by
\eqref{eq:hratio-bound}. Thus Lemma~\ref{lem:lambert-periodization}
applies with $c=\omega^k$, $\beta=k/d$ and $\Phi=\mathcal G_{Z^{(k)}}$. We write
$\mathcal E^{(k)}:=\mathcal E_{\Phi}=\mathcal E_{Z^{(k)}}$ and
$\rho^{(k)}:=\rho_{\Phi}$.

The next lemma computes the residue function $\rho^{(k)}$ and shows that it does not vanish for small $|q|$.

\begin{lemma}\label{lem:residue-k}
For $1\le k\le d-1$, we have
\begin{equation}\label{eq:rho-k}
 \rho^{(k)}(\tau)=x^{-2}\sum_{i\ge0}\omega^{ki}q^{(m+1)i^2}
 \frac{(\omega^kq;q)_i}{(q^d;q^d)_i}.
\end{equation}
The function $x^2\rho^{(k)}$ is holomorphic on $|q|<1$, independent of
$x$, and equal to $1+O(q)$ as $q\to0$. Moreover,
$|x^2\rho^{(k)}(\tau)-1|\le3/7$ whenever $0<|q|\le1/2$. Consequently, the set
\[
 \Xi_m:=\bigcup_{k=1}^{d-1}\{q:\ 0<|q|<1,\ \rho^{(k)}(\tau)=0\}
\]
is a discrete subset of $\{q:0<|q|<1\}$, depends only on $m$, and is
disjoint from $\{q:0<|q|\le1/2\}$.
\end{lemma}

\begin{proof}
With $c=\omega^k$ and $b_i=b_i(\omega^k)$ in \eqref{eq:EPhi-def}, we have
\[
 x^{2i}c^{2mi}q^{mi^2}b_i(\omega^k)
 =x^{2i}\omega^{2kmi}q^{mi^2}\Bigl(\frac{\omega^k}{x}\Bigr)^{2i}
  q^{i^2}h_i(\omega^k)
 =\omega^{2k(m+1)i}q^{(m+1)i^2}h_i(\omega^k),
\]
and $\omega^{2(m+1)}=\omega^{d+1}=\omega$. This gives
\eqref{eq:rho-k}. The series \eqref{eq:rho-k} converges normally on
$|q|\le Q$ for every $Q<1$ by \eqref{eq:hratio-bound}, its terms are
holomorphic on $|q|<1$, and its $i=0$ term is $1$. Hence $x^2\rho^{(k)}$
is holomorphic on $|q|<1$ with value $1$ at $q=0$. For $0<|q|\le1/2$,
the inequalities $(m+1)i^2\ge3i^2\ge3i$ and \eqref{eq:hratio-bound} with
$Q=1/2$ give
\[
 |x^2\rho^{(k)}(\tau)-1|
 \le B'_{1/2}\sum_{i\ge1}|q|^{3i}
 \le B'_{1/2}\,\frac{1/8}{1-1/8}=\frac{B'_{1/2}}{7}.
\]
Since $1+t\le e^t$ and $(1-t)^{-1}\le1+2t\le e^{2t}$ for $0\le t\le1/2$,
and $d\ge5$, we have
\[
 B'_{1/2}=\prod_{l\ge1}\frac{1+2^{-l}}{1-2^{-dl}}
 \le\exp\Bigl(\sum_{l\ge1}2^{-l}+2\sum_{l\ge1}2^{-dl}\Bigr)
 =\exp\Bigl(1+\frac{2}{2^d-1}\Bigr)\le e^{1+2/31}<3,
\]
and hence $|x^2\rho^{(k)}(\tau)-1|\le3/7<1$ and $\rho^{(k)}(\tau)\ne0$ for
$0<|q|\le1/2$. Finally, $x^2\rho^{(k)}$ is a holomorphic function on the
disc $|q|<1$ that is not identically zero, and so its zero set is discrete
in $|q|<1$. The union of these $d-1$ zero sets, intersected with
$0<|q|<1$, is $\Xi_m$. Therefore $\Xi_m$ is discrete, independent of $x$,
and disjoint from $0<|q|\le1/2$. This proves the lemma.
\end{proof}

With these lemmas in hand, we now prove Theorem~\ref{thm:remark-proof}.

\begin{proof}[Proof of Theorem~\ref{thm:remark-proof}]
(i) By Lemma~\ref{lem:solution-space}(c), the sequence
$X^{(m)}(x;q)/G_d(x;q)=(S_j(x))_{j\ge0}$ lies in $V_1$, and by \eqref{eq:S-asymp}
with $c=x$ and $\lambda=1$, it tends to $1$. More precisely, we have
$X_j^{(m)}(x;q)/G_d(x;q)=1+O(|q|^{j+1})$. Conversely, let
$Y\in V_\gamma$ with $Y_j\to1$. Letting $j\to\infty$ in
\eqref{eq:Y-recurrence}, we obtain $\sum_{r=0}^{d-1}x^rY_{j+r-2}\to C_x$
and $q^{j-1}Y_j\to0$. Hence $C_x\gamma=C_x$. If $C_x\ne0$, then
$\gamma=1$, and by \eqref{eq:Y-decomposition} the sequence
$Y-X^{(m)}(x;q)/G_d(x;q)=\sum_ks_kZ^{(k)}$ tends to $0$. By
\eqref{eq:S-asymp}, this gives $\sum_ks_k\lambda_k^j\to0$. Hence all $s_k$ vanish by
the uniqueness theorem \cite[Chapter~I]{Besicovitch} invoked in the proof of Lemma~\ref{lem:solution-space}(b), and $Y=X^{(m)}(x;q)/G_d(x;q)$. If $C_x=0$,
then $x=\omega^{k_0}$ and $\lambda_{k_0}=1$, and
Lemma~\ref{lem:solution-space}(c) gives $Y=\sum_ks_kZ^{(k)}$ with
$Z^{(k_0)}=X^{(m)}(x;q)/G_d(x;q)$. Now $Y_j\to1$ and \eqref{eq:S-asymp}
give $(s_{k_0}-1)\lambda_{k_0}^j+\sum_{k\ne k_0}s_k\lambda_k^j\to0$, so
the same uniqueness theorem yields $s_{k_0}=1$ and $s_k=0$ for
$k\ne k_0$. Again, $Y=X^{(m)}(x;q)/G_d(x;q)$. This proves (i).

(ii) Let $Y\in V_\gamma$. By Lemma~\ref{lem:solution-space}(c), we have
$Y=\gamma'X^{(m)}(x;q)/G_d(x;q)+\sum_{k=1}^{d-1}s_kZ^{(k)}$, where
$\gamma':=\gamma$ if $C_x\ne0$ and $\gamma':=0$ if $C_x=0$. The seed $\mathcal G_Y$ depends linearly on $Y$, and so by
\eqref{eq:S-seed} with $c=x$ (which, by
Lemma~\ref{lem:lambert-solutions}(c), reproduces
\eqref{eq:seed-fractions} divided by $G_d(x;q)$) and with $c=\omega^k$,
\begin{equation}\label{eq:GY-decomposition}
 \mathcal G_Y(u)
 =\frac{\gamma'}{G_d(x;q)}\,\mathcal G_{d,x}(u;q)
 +\sum_{k=1}^{d-1}s_k\mathcal G_{Z^{(k)}}(u)\qquad(|u|<1),
\end{equation}
and the right-hand side is meromorphic on $\CC$ with at most simple
poles at the points $x^{-1}q^{-i}$ and $\omega^{-k}q^{-i}$, $i\ge0$.
This proves (ii)(a).
The periodization is linear in the seed and \eqref{eq:GY-decomposition}
is a finite sum, so
\begin{equation}\label{eq:EY-decomposition}
 \mathcal E_Y=\frac{\gamma'}{G_d(x;q)}\,\mathcal E_{m,x}
 +\sum_{k=1}^{d-1}s_k\mathcal E^{(k)},
\end{equation}
where the coefficients $\gamma'/G_d(x;q)$ and $s_k$ are constants ($q$
is fixed). Moreover, each of the functions $\mathcal E_{m,x}$ and
$\mathcal E^{(k)}$ converges normally on compact subsets of the
complement of its orbit, namely $\mathcal D_x$ and
$\mathcal D_{\omega^k}$, respectively, extends jointly meromorphically with at most
simple poles on that orbit, and satisfies \eqref{eq:twisted-law}, by
Theorem~\ref{thm:elliptic} and Lemma~\ref{lem:lambert-periodization}(a).
For the fixed $\tau$, \eqref{eq:EY-decomposition} is therefore an
identity between meromorphic functions of $w$. Hence (ii)(b) follows.

We now prove (ii)(c). Suppose that $Y_0,\dots,Y_{d-2}$ and $\gamma$ are holomorphic
functions of $q$ on an open set $\Omega\subset\{q:0<|q|<1\}$. Each
$S_j(c)$ in \eqref{eq:S-def} is holomorphic in $q$ on $|q|<1$. Indeed, the
factors $h_i(c)$ are rational functions of $q$ whose denominators
$(c^dq^d;q^d)_i$ do not vanish there, and the series converges normally
on compact subsets of $|q|<1$ by \eqref{eq:hratio-bound}. Hence the
entries of the $(d-1)\times(d-1)$ matrix
$M(q):=\bigl(Z^{(k)}_j\bigr)_{0\le j\le d-2,\,1\le k\le d-1}$ and the
numbers $X^{(m)}_j(x;q)/G_d(x;q)=S_j(x)$, $0\le j\le d-2$, are
holomorphic on $|q|<1$, and $M(q)$ is invertible for every $q$ by
Lemma~\ref{lem:solution-space}(a),(b). Comparing initial values in the
decomposition of Lemma~\ref{lem:solution-space}(c) gives
\[
 (s_1,\dots,s_{d-1})^{T}
 =M(q)^{-1}\Bigl(Y_j-\gamma'\,\frac{X^{(m)}_j(x;q)}{G_d(x;q)}\Bigr)_{0\le j\le d-2}^{T},
\]
and therefore the coefficients $s_k$ and $\gamma'/G_d(x;q)$ in
\eqref{eq:EY-decomposition} are holomorphic functions of $\tau$ on
$\widetilde\Omega$ (recall that $G_d(x;q)\ne0$). Since
\eqref{eq:EY-decomposition} is a finite sum of products of holomorphic
functions of $\tau$ with functions that are jointly meromorphic on
$\HH\times\CC$ with at most simple poles on the orbits, (ii)(c) follows. We note that, without such a hypothesis, no
joint regularity can be expected. For instance, for $1\le k\le d-1$, the
sequence $Y:=\overline{q}\,Z^{(k)}$ lies in $V_0$ for every $q$, but
$\mathcal E_Y=\overline{q}\,\mathcal E^{(k)}$ is not meromorphic in
$\tau$.

We now prove (ii)(d), the characterization by the location of the poles. Let $\Xi_m$ be the set of Lemma~\ref{lem:residue-k} and assume that
$q\notin\Xi_m$. We put $K:=\{k:1\le k\le d-1,\ \omega^k\ne x\}$. For
$k\in K$, the point $w_k:=-k/d$ lies on $\mathcal D_{\omega^k}$ but on
none of the other orbits occurring in \eqref{eq:EY-decomposition}, since
$\mathcal D_x$ and the $\mathcal D_{\omega^{k'}}$, $k'\ne k$, are
disjoint from $\mathcal D_{\omega^k}$. Hence only the term
$s_k\mathcal E^{(k)}$ can be singular at $w_k$, and \eqref{eq:res-Phi}
gives $\Res_{w=w_k}\mathcal E_Y(\tau,w)=-s_k\rho^{(k)}(\tau)/(2\pi i)$.
Since $\rho^{(k)}(\tau)\ne0$, the function $\mathcal E_Y(\tau,\cdot)$ has a
pole at $w_k\notin\mathcal D_x$ if and only if $s_k\ne0$. Therefore, if
the poles of $\mathcal E_Y(\tau,\cdot)$ are contained in $\mathcal D_x$,
then $s_k=0$ for all $k\in K$, and \eqref{eq:GY-decomposition} shows
that $Y=\gamma'X^{(m)}(x;q)/G_d(x;q)$ if $C_x\ne0$, and
$Y=s_{k_0}Z^{(k_0)}=s_{k_0}X^{(m)}(x;q)/G_d(x;q)$ if $C_x=0$ and
$x=\omega^{k_0}$. In both cases, $Y$ is a scalar multiple of
$X^{(m)}(x;q)$. Conversely, if $Y=cX^{(m)}(x;q)$ with $c\in\CC$, then
$\mathcal G_Y=c\,\mathcal G_{d,x}$ and $\mathcal E_Y=c\,\mathcal E_{m,x}$,
whose poles lie on $\mathcal D_x$ by Theorem~\ref{thm:elliptic}(i).
Finally, let $k\in K$ and $Y:=Z^{(k)}\in V_0$. Then $Y$ is not a scalar
multiple of $X^{(m)}(x;q)$, since $cX_j^{(m)}(x;q)$ converges as
$j\to\infty$ for every $c\in\CC$, whereas
$Z^{(k)}_j=\lambda_k^j+o(1)$ with $|\lambda_k|=1$ and $\lambda_k\ne1$
does not. Its periodization $\mathcal E_Y=\mathcal E^{(k)}$ has poles
only on $\mathcal D_{\omega^k}$. Since $\rho^{(k)}(\tau)\ne0$,
Lemma~\ref{lem:lambert-periodization}(c) shows that every point of this
orbit is a simple pole. This proves (ii)(d) and completes the proof of (ii).

(iii) Let $x=1$, so that $C_x=d$, and put
\[
 Y^*_j:=\frac{X_j^{(m)}(1;q)}{G_d(1;q)}
 =\sum_{n\ge0}q^{n^2+jn}\frac{(q;q)_n}{(q^d;q^d)_n}\in\ZZ[[q]]
 \qquad(j\ge0),
\]
where the second equality holds by Lemma~\ref{lem:lambert-solutions}(c). By
Lemma~\ref{lem:solution-space}(c), we have $Y^*\in V_1$. That is, $Y^*$ satisfies
\eqref{eq:Y-recurrence} with $\gamma=1$ and $C_x\gamma=d$, and these
identities hold in $\ZZ[[q]]$. Moreover, $Y^*_0=\RRm_m(\zeta_d;q)$ by
\eqref{eq:seed-specialization}, and we have
\begin{equation}\label{eq:Ystar-congruence}
 Y^*_j\equiv1\pmod{q^{j+1}}\qquad(j\ge0),
\end{equation}
because every term with $n\ge1$ is divisible by $q^{n^2+jn}$ and
$n^2+jn\ge j+1$. The backward procedure \eqref{eq:backward} uses only
ring operations in $\ZZ[q]$, so $Y_j\in\ZZ[q]$ for $0\le j\le J+d-2$.
In particular, $P_J=Y_0\in\ZZ[q]$. We claim that
$Y_j\equiv Y^*_j\pmod{q^{J+1}}$ in $\ZZ[[q]]$ for $0\le j\le J+d-2$. We
prove the claim by descending induction on $j$. For
$J\le j\le J+d-2$, we have $Y_j=1\equiv Y^*_j\pmod{q^{j+1}}$ by
\eqref{eq:Ystar-congruence}, and $j+1\ge J+1$. For $2\le j\le J+1$,
assume the congruence for all indices $j-1,\dots,J+d-2$, which include
$j,j-1,\dots,j+d-3$ since $j+d-3\le J+d-2$. Subtracting
\eqref{eq:Y-recurrence} for $Y^*$ at this $j$ from \eqref{eq:backward}
gives
\[
 Y_{j-2}-Y^*_{j-2}
 =q^{j-1}\bigl(Y_j-Y^*_j\bigr)
 -\sum_{r=1}^{d-1}\bigl(Y_{j+r-2}-Y^*_{j+r-2}\bigr)
 \equiv0\pmod{q^{J+1}}.
\]
This is the congruence for $j-2$, and the claim follows. The case $j=0$ of the claim gives
$P_J=Y_0\equiv Y^*_0=\RRm_m(\zeta_d;q)\pmod{q^{J+1}}$. (For $J=0$, no
backward step is needed, and $P_0=1$.)
Thus the normalized recurrence, together with the terminal values $1$,
determines the primitive rank specialization $\RRm_m(\zeta_d;q)$ to any
prescribed $q$-adic precision. This proves (iii) and completes the proof of the theorem.
\end{proof}

\section{The first new case: \texorpdfstring{$m=2$}{m=2} and level 5}
\label{sec:example}

We conclude by working out the first genuinely higher case, $m=2$. The purpose of this section is illustrative. We write down the finite recurrence and the defect polynomial explicitly, and we describe the bilateral periodization of the seed.
Subtracting the corresponding Appell--Lerch polar term gives
a holomorphic finite part with four theta components. 

We set
$m=2$, $d=5$, and $\omega=e^{2\pi i/5}$. The relevant generating functions are
\[
\RRm_2(\zeta_5;q)=\sum_{k\ge0}q^{k^2}\frac{(q;q)_k}{(q^5;q^5)_k},
\qquad
G_5(x;q)=\frac{(x^5q^5;q^5)_\infty}{(xq;q)_\infty},
\qquad
X_j^{(2)}(x;q)=\sum_{n\geq 0}q^{n^2+jn}G_5(xq^n;q).
\]

Writing $X_j=X_j^{(2)}(1;q)$, the recurrence \eqref{eq:recurrence} reads
\[
 X_{j-2}+X_{j-1}+X_j+X_{j+1}+X_{j+2}-q^{j-1}X_j
       =5\frac{(q^5;q^5)_\infty}{(q;q)_\infty}.
\]
The corresponding defect polynomial \eqref{eq:R-explicit} is the cubic
\begin{equation*}%\label{eq:R-level5}
 \begin{split}
 R_{5,1}(u;q)={}&(q^{-1}u^2-u^3-u^2-u-1)X_0
                  -(u^2+u)X_1\\
               &-(u^3+u^2)X_2-u^3X_3.
 \end{split}
\end{equation*}
The periodization and its polar kernel are
\[
 \mathcal E_2(\tau,w)=\sum_{n\in\ZZ}q^{2n^2}u^{4n}
                          \mathcal G_{5,1}(q^nu;q),\qquad
 \mathcal K_2(\tau,w)=\sum_{n\in\ZZ}\frac{q^{2n^2}u^{4n}}{1-uq^n}.
\]
They satisfy the index-two elliptic law and have possible poles
on $\ZZ\tau+\ZZ$. The four theta coefficients are
\begin{equation*}%\label{eq:level5-coefficients}
 h_{2,\ell}(\tau)=q^{-\ell^2/8}
 \sum_{k\ge1}\frac{(q^{5(k+1)};q^5)_\infty}{(q^{k+1};q)_\infty}
       \sum_{s=1}^{k}q^{3k^2-2s^2+\ell s},\qquad 0\le\ell\le3
\end{equation*}
with leading terms 
\[
 \begin{aligned}
 h_{2,0}&=q(1+O(q^2)),&
 h_{2,1}&=q^{15/8}(1+O(q^2)),\\
 h_{2,2}&=q^{5/2}(1+O(q^2)),&
 h_{2,3}&=q^{23/8}(1+O(q^2)).
 \end{aligned}
\]
In particular, each of the four theta components occurs.

\subsection{The other fifth-root deformations}

Choose $1\le a\le4$, and put
$x=\omega^a$. Then $C_x=0$, so the inhomogeneous term in the
recurrence vanishes. With $X_j=X_j^{(2)}(x;q)$, the recurrence is
\[
 X_{j-2}+xX_{j-1}+x^2X_j+x^3X_{j+1}+x^4X_{j+2}
     -q^{j-1}X_j=0
\]
with boundary values
\[
 X_{-1}=-xX_0+(1-x^2)X_1-x^3X_2-x^4X_3,\qquad
 X_{-2}=q^{-1}X_0-xX_1+X_3.
\]
We substitute into \eqref{eq:R-explicit} and obtain the cubic
\begin{equation*}%\label{eq:R-level5-twisted}
 \begin{split}
 R_{5,x}(u;q)={}&(q^{-1}x^{-2}u^2-u^3-u^2-u-1)X_0\\
 &+(x^{-1}u^3-xu^3-xu^2-xu)X_1\\
 &-x^2(u^3+u^2)X_2-x^3u^3X_3.
 \end{split}
\end{equation*}
The functional equation becomes
\[
 \mathcal G_{5,x}(qu;q)
   =x^2q\frac{u^5-1}{u^2(u-1)}\mathcal G_{5,x}(u;q)
        +qu^{-2}R_{5,x}(u;q)
\]
and the rational defect in Proposition~\ref{prop:partial-fractions} is
\[
 Q_{5,x}(u;q)=(1-u)R_{5,x}(u;q),\qquad
 D_{r,x}(q)=\frac{\omega^r}{5}(1-\omega^{-r})
                                  R_{5,x}(\omega^{-r};q).
\]
In particular, $D_{0,x}=0$. 

In this case, we have
\[
 g_{k,x}(q)=\frac{(q^{5(k+1)};q^5)_\infty}{(xq^{k+1};q)_\infty},
 \qquad
 \mathcal E_{2,x}(\tau,w)
    =\sum_{n\in\ZZ}x^{-2n}q^{2n^2}u^{4n}
                         \mathcal G_{5,x}(q^nu;q)
\]
with possible poles on $-a/5+\Lambda_\tau$, where $\Lambda_\tau\coloneqq \mathbb{Z}\tau+\mathbb{Z}$, and
\[
 \mathcal E_{2,x}(\tau,w+\tau)
     =x^2q^{-2}u^{-4}\mathcal E_{2,x}(\tau,w).
\]
Since $x^5=1$, the residue function and the coefficients simplify to
\begin{align*}
 \rho_{2,x}(\tau)
     &=x^{-2}\sum_{k\ge0}x^k q^{3k^2}g_{k,x}(q),\\
                     %    \label{eq:rho-level5-twisted}\\
 C_{2,\ell,x}(\tau)
     &=x^{\ell-2}\sum_{k\ge1}g_{k,x}(q)
          \sum_{s=1}^{k}x^{k-s}q^{3k^2-2s^2+\ell s},
                  \qquad 0\le\ell\le3.%\label{eq:C-level5-twisted}
\end{align*}
Translation by $a/10$ removes the twist and moves the possible pole
orbit to $-3a/10+\Lambda_\tau$. The standard theta coefficients are
\[
 h_{2,\ell,x}(\tau)=e^{\pi i\ell a/5}q^{-\ell^2/8}
                               C_{2,\ell,x}(\tau)
\]
with leading terms
\[
 h_{2,\ell,x}(\tau)
   =e^{\pi i\ell a/5}x^{\ell-2}
          q^{1+\ell-\ell^2/8}(1+O(q^2)).
\]
Thus all four coefficient functions occur for every nontrivial fifth-root
deformation as well, whereas positivity is special to $x=1$.

\section*{Lean Certificate and AI Declaration}

This work grew out of a human--AI collaboration. AxiomProver assisted
with mathematical exploration and produced the Lean formalization.
An AI language model assisted with exposition.
The human authors wrote and reviewed the paper and are responsible
for its mathematical content and presentation.

The formalization is available at
\begin{center}
\url{https://github.com/AxiomMath/HigherDyson3}.
\end{center}
The repository uses Lean~4.34.0 and Mathlib.
Its definitions encode the Dyson symbols, generating series, product,
seed, periodization, residue function, polar kernel, and theta
coefficients used in this paper. The seed is defined by
\eqref{eq:seed-fractions}, whose identification with
\eqref{eq:seed-def} is established in Lemma~\ref{lem:seed}.
The configuration \texttt{Comparator/comparator.json} lists
32 certificate targets. These are organized as follows:
\begin{itemize}[leftmargin=*,itemsep=3pt]
\item \textbf{Theorem~\ref{thm:finite}:}
the finite recurrence and functional equation, uniqueness and degree
of the defect polynomial, its constant term and behavior at $q=0$,
and the vanishing of $C_x$ at nontrivial $d$th roots of unity.

\item \textbf{Theorem~\ref{thm:elliptic}:}
normal convergence, holomorphy off the divisor, at-most-simple poles,
the twisted and untwisted elliptic laws, the residue formula,
holomorphy and the leading term of $\rho_{m,x}$, genuine and removable
poles, and nonvanishing of the periodization.

\item \textbf{Theorem~\ref{thm:theta}:}
holomorphic extension of the finite part, its twisted and untwisted
elliptic laws, the explicit theta expansion and its uniqueness,
normal convergence and holomorphy of the coefficients, their leading
terms and nonvanishing, nonvanishing of the finite part, and
positivity on the positive imaginary axis when $x=1$.

\item \textbf{Theorem~\ref{thm:combinatorialgenseries}:}
the full-rank generating identity and the symmetry
$N_m(r,N)=N_m(-r,N)$.

\item \textbf{Theorems~\ref{rem:recurrence-implies}
and~\ref{thm:remark-proof}:}
uniqueness of the normalized Dyson sequence among recurrence
solutions tending to $1$, convergence and meromorphic continuation
of every solution's seed, and, for fixed $\tau$, normal convergence,
holomorphy off the prescribed orbits, at-most-simple poles, and the
twisted elliptic law of its periodization.
The bundled target also includes discreteness of the exceptional set and
its disjointness from $0<|q|\le1/2$, characterization of the Dyson
sequence by the location of the poles, non-Dyson solutions with a
different single orbit of genuine simple poles, and the
$q$-adic congruence and coefficient agreement for the backward
polynomials $P_J\in\ZZ[q]$.
\end{itemize}

The repository also formalizes supporting results, including the
Lambert-type solutions, the dimension of the homogeneous solution
space, its eigen-solution basis, and the residue formulas on the
torsion orbits. The holomorphic-family assertion of
Theorem~\ref{thm:remark-proof}(ii)(c) is proved separately as
\texttt{solution\_periodization\_family} in
\nolinkurl{HigherDyson3/Characterization/SolutionFamilies.lean}.
These supporting results are distinct from the list of
32 comparator targets.

The file \texttt{Challenge/Basic.lean} specifies the certificate
statements, and \texttt{Solution/Basic.lean} supplies their proofs
through the results imported from \texttt{HigherDyson3}.
The \texttt{sorry} placeholders in the challenge file specify the
statements to be checked and are not proofs used by the solution.
The comparator checks the solution against these statements and
permits only the foundational axioms \texttt{propext},
\texttt{Quot.sound}, and \texttt{Classical.choice}, with no additional
mathematical axioms. Mathlib supplies previously formalized results
used by the proofs.

\end{document}